\documentclass[a4paper, reqno, 11pt]{amsart}
\usepackage[utf8]{inputenc}
\usepackage{amsmath}
\usepackage{amssymb}
\usepackage{amsthm}
\usepackage[dvipsnames]{xcolor}
\usepackage{comment}

\newtheorem{theorem}{Theorem}[section]

\newtheorem{lemma}[theorem]{Lemma}

\newtheorem{proposition}[theorem]{Proposition}
\newtheorem{remark}[theorem]{Remark}
\theoremstyle{definition}
\newtheorem{definition}{Definition}[section]

\DeclareMathOperator\supp{supp}

\usepackage[luatex,colorlinks=true, linkcolor=WildStrawberry, citecolor=PineGreen]{hyperref}



\title[A pathological set for a.s. properties of Gaussian measures]{A pathological set regarding the propagation of almost sure properties of Gaussian measures}
\author{Pablo Merino}
\address[P. Merino]{Basque Center for Applied Mathematics, 48009 Bilbao, Basque Country}
\date{}

\begin{document}
	\maketitle
	
	\begin{abstract}
	We provide a complementary result to the quasi-invariance of Gaussian measures supported on Sobolev spaces with high regularity under the dynamics of the three-dimensional periodic defocusing nonlinear wave equation from Gunaratnam-Oh-Tzvetkov-Weber (2022). Namely, given $p > 2$, $\sigma$ large enough and $\alpha \in [\max\{\frac{1}{2} - \frac{3}{p} , 0\} , 1 - \frac{2}{p})$, we prove the existence of dense subsets of the Sobolev space $W^{\sigma + \alpha,p}(\mathbb{T}^3)$ which do not preserve the $L^p$-based regularity $\sigma$ throughout the aforementioned dynamics. This is in sharp contrast with the propagation under the flow of almost sure properties.
	\end{abstract}
	

	\section{Introduction}\label{sec:intro}
	Let $p \neq 2$ and $\sigma \geq 0$ large enough. We consider the Cauchy problem associated to the defocusing cubic nonlinear wave equation on $\mathbb{T}^3 = (\mathbb{R}/\mathbb{Z})^3$:
		\begin{equation}\label{eq:cauchywave}
			\left\{
			\begin{array}{l}
				\partial_{tt} u - \Delta u + u^3 = 0,
				\\
				u(x,0) = 0,  \ \ \partial_tu(x,0) = \langle D \rangle g(x),
			\end{array}
			\right.
		\end{equation}
		where $u:\mathbb{T}^3 \times \mathbb{R} \rightarrow \mathbb{R}$ is the unknown function, $D = \frac{1}{i} \nabla$ and $\langle D \rangle = (1-\Delta)^{\frac{1}{2}}$. The initial datum $g$ in \eqref{eq:cauchywave} is an element of the set
	\begin{align}\label{eq:family-S}
	\Theta_{\mathcal{H}} =  \{\delta \langle D \rangle^{-\sigma} \mathcal{H} : \delta > 0\} + C^{\infty}(\mathbb{T}^3),
	\end{align}
	where $\mathcal{H}$ is a function in $W^{\alpha,p}(\mathbb{T}^3)$ which along $\sin(t|D|)$ ceases to belong to $L^p(\mathbb{T}^3)$, for $t \neq 0$, $\alpha$ and $p$ as in Theorem \ref{thm:nonpropag-linear}.  Given $\beta \geq 0$, $d \in \mathbb{N}$ and $q \in [1,\infty]$, we denote by $W^{\beta,q}(\mathbb{T}^d)$ (resp. $W^{\beta,q}(\mathbb{R}^d)$) the set of distributions on $\mathbb{T}^d$ (resp. on $\mathbb{R}^d$) which weak derivatives of order up to $\beta$ belong to $L^q(\mathbb{T}^d)$ (resp. $L^q(\mathbb{R}^d)$), endowed with the norm
        \begin{align*}
            &\|f\|_{W^{\beta,q}(\mathbb{T}^d)} := \|(1-\Delta)^{\frac{\beta}{2}}f\|_{L^q(\mathbb{T}^d)}, \\
            &(\text{resp. } \|f\|_{W^{\beta,q}(\mathbb{R}^d)} := \|(1-\Delta)^{\frac{\beta}{2}}f\|_{L^q(\mathbb{R}^d)}).
        \end{align*} 
        From the density of $C^{\infty}(\mathbb{T}^3)$ in $W^{s,p}(\mathbb{T}^3)$ for any $s \geq 0$ it follows that $\Theta_{\mathcal{H}}$ is a dense subset of $W^{\sigma + \alpha,p}(\mathbb{T}^3)$. 
        
        From now on, generic spatial dimensions are considered as $d \geq 2$. The linear evolution of the equation in \eqref{eq:cauchywave} extended to the space dimension $d$ is defined as
        \begin{align*}
        S(t)(&u,v) := \cos(t|D|) u + \frac{\sin(t|D|)}{|D|} v,
        \end{align*}
        where
        \begin{align*}
        &\cos(t|D|) u = \sum_{k \in \mathbb{Z}^d} \cos(t|k|) \widehat{u}(k) e^{ikx}, \nonumber \\
        &\frac{\sin(t|D|)}{|D|} v =  t\widehat{v}(0) + \sum_{k \in \mathbb{Z}^d \setminus \{ 0 \} } \frac{\sin(t|k|)}{|k|} \widehat{v}(k) e^{ikx},
        \end{align*}
        for any $u,v \in C^{\infty}(\mathbb{T}^d)$, and where we used the notation $|D| = \sqrt{-\Delta}$. Thus, with the choice of data given in \eqref{eq:cauchywave}, the homogeneous part of the solution is  formally given by
\begin{align}\label{eq:linearflow-flambda}
&S(t)(0,\langle D \rangle g) = t\widehat{g}(0) + \sum_{k \in \mathbb{Z}^d \setminus \{ 0 \} } \frac{\sin(t|k|)}{|k|} \langle k \rangle \widehat{g}(k) e^{ikx}.
\end{align}

\hfill \break

In works \cite{Sj70}, \cite{Peral1980} and \cite{Miyachi1980}, the authors proved that, for fixed nonzero times, $L^p$ estimates on $\mathbb{R}^d$ exhibit a loss of derivatives under the linear wave evolution. Namely, they showed that, given $t \neq 0$ fixed,
\begin{align}\label{eq:bounded-lp}
\langle D \rangle^{-\sigma}e^{it|D|} \text{ is bounded on } L^p(\mathbb{R}^d) \Longleftrightarrow \sigma \geq (d-1) \Big| \frac{1}{p} - \frac{1}{2} \Big|.
\end{align}
Regarding the implication $\Longrightarrow$ of this statement, in \cite{Sj70} and \cite{Peral1980} the authors used elements of the form of $F_{\lambda}$ defined in \eqref{def:g-lambda} as counterexample for $L^p$-based derivative regularity propagation, for $1 < p < 2$. 

Our goals are to extend this loss of $L^p$-based regularity to the corresponding periodic nonlinear flow with $d = 3$ and $p > 2$, and to prove that $\Theta_{\mathcal{H}}$ is thus a dense subset of $W^{\sigma + \alpha,p}(\mathbb{T}^3)$ which does not preserve $W^{\sigma,p}$ regularity under the flow of \eqref{eq:cauchywave}, when we choose $p$, $\alpha$ and $\sigma$ as in Theorem \ref{thm:nonpropag-nonlinear}.  This loss essentially follows from the homogeneous part of the solution in \eqref{eq:cauchywave}.

Namely, after giving explicitly the aforementioned counterexample on $\mathbb{R}^d$ for $1 < p < 2$ and the linear flow, we use periodization and duality arguments that allow us to extend this non-propagation of regularity to the periodic linear flow and $2 < p < \infty$, respectively, as well as a suitable control of the Duhamel integral in the $W^{\sigma,p}$ norm. Notice that, if $p = 2$, we recover the usual conservation of energy. A deeper analysis of this family of functions $F_{\lambda}$, and of its periodic counterparts $f_{\lambda}$ defined in \eqref{eq:flambdadef-1}, can be found in \cite{W65}. 

This is pathological in the sense that a suitable randomization of initial data (see \eqref{eq:randomid}) ensures this propagation almost surely with respect to the Gaussian measure induced by such random data, as it was proved in \cite{GOTW22} by showing the quasi-invariance of this measure under the flow of \eqref{eq:cauchywave} for data regular enough.


Given $t \in(-1/2,1/2) \setminus \{0\}$, once we prove the existence of $\mathcal{H}$ with loss of $L^p$-based regularity for $S(t)(0,\langle D \rangle \mathcal{H})$ in Theorem \ref{thm:nonpropag-linear} below, in order to extend this behaviour to the three-dimensional nonlinear flow it suffices to bound the quantity (for some suitable $T > 0$)
	\begin{align}\label{eq:quantity-to-control}
		\left\| \int_0^t \frac{\sin((t-\tau)|D|)}{|D|} u^3(\tau) d\tau \right\|_{L^{\infty}_T W^{\sigma,p}(\mathbb{T}^3)}.
	\end{align} 
	
	Namely, using Strichartz estimates, we look for an upper bound which depends essentially on the $H^r(\mathbb{T}^3)$ norm of the solution, for values of $r$ such that the initial datum is in $\vec{H}^r(\mathbb{T}^3) := H^r(\mathbb{T}^3) \times H^{r-1}(\mathbb{T}^3)$ ensuring a well-defined global dynamics (see for instance \cite{GST19}). 
	
All in all, this work deals with deterministic analysis, although the motivation (given in  Section \ref{sec:prob}) is deeply probabilistic. Our main contribution is the following.

\begin{theorem}\label{thm:nonpropag-nonlinear}
Let $p > 2$, $\alpha \in  [\max\{\frac{1}{2} - \frac{3}{p} , 0\} ,  1 - \frac{2}{p})$ and $\sigma \in \mathbb{R}$ such that $\sigma + \alpha > \frac{3}{2}$. Let $t \in (-1/2,1/2) \setminus \{0\}$ and $\mathcal{H}$ the function from Theorem \ref{thm:nonpropag-linear} below. Then $\Theta_{\mathcal{H}}$ is a dense subset of $W^{\sigma + \alpha,p}(\mathbb{T}^3)$ satisfying the following property:\\

For any $g \in \Theta_{\mathcal{H}}$, the Cauchy problem \eqref{eq:cauchywave} admits a unique global solution $u$ in $C(\mathbb{R},H^{\sigma + \alpha}(\mathbb{T}^3))$ such that
\begin{align*}
u(t,\cdot) \notin W^{\sigma,p}(\mathbb{T}^3).
\end{align*}
\end{theorem}
	
As we mentioned before, Theorem \ref{thm:nonpropag-nonlinear} will essentially follow from the next result, which is based on the analysis of the linear flow of \eqref{eq:cauchywave} on $\mathbb{R}^d$, a suitable periodization and a duality argument. Regarding the Euclidean linear case, we are only making explicit the analysis from \cite{Sj70}, \cite{Peral1980} and \cite{Miyachi1980}.
	\begin{theorem}\label{thm:nonpropag-linear}
	Let $p \in (1,2) \cup (2,\infty)$ and $0 \leq \alpha < (d-1)\left|\frac{1}{2} - \frac{1}{p} \right|$. Then, for any $t \in (-1/2,1/2) \setminus \{0\}$ there exists some  $\mathcal{H} \in W^{\alpha,p}(\mathbb{T}^d)$ such that
\begin{align*}
S(t)(0 , \langle D \rangle \mathcal{H}) \notin L^p(\mathbb{T}^d).
\end{align*} 
	\end{theorem}

\begin{remark}
We emphasize that, for the case $2 < p < \infty$ in Theorem \ref{thm:nonpropag-linear}, each function $\mathcal{H}$ and each set $\Theta_{\mathcal{H}}$ given in \eqref{eq:family-S} depend on the choice of $t \in (-1/2,1/2) \setminus \{0\}$. This is enough to show the aforementioned pathological behaviour in this range of $p$ when it is compared with the probabilistic analysis in Section \ref{sec:prob}. 

As it is shown in Sections \ref{sec:linearcex} and \ref{sec:localization}, in the case $1 < p < 2$ we find a stronger result. Namely, given $\alpha$ as in Theorem \ref{thm:nonpropag-linear} there exists some $\mathcal{H} \in W^{\alpha,p}(\mathbb{T}^d)$ such that $S(t)(0,\langle D \rangle \mathcal{H}) \notin L^p(\mathbb{T}^d)$ for any $t \in (-1/2,1/2) \setminus \{0\}$.
\end{remark} 

\begin{remark}
The bound $1/2$ for $|t|$ in Theorems \ref{thm:nonpropag-nonlinear} and \ref{thm:nonpropag-linear} is a technical constraint due to the way we apply the finite speed of propagation and the method we use in the proof of Theorem \ref{thm:divergence-torus} to show the divergence of
\begin{align*}
\| S(t)(0, \langle D \rangle f_{\lambda})\|_{W^{-\alpha,p}(\mathbb{T}^d)},
\end{align*}
for suitable values of $\lambda$, $\alpha$ and $p$. Although they may be extended to any $t$, we admit this limitation since it does not prevent us from illustrating the aforementioned pathological behaviour. 
\end{remark}

\begin{remark}
This work shares the same philosophy than others in which the conceptual contrast between topological and measure-theoretic properties around the flow of certain PDEs is studied. We refer to \cite{ST20}, where they present the existence of a dense set of the Sobolev space of super-critical regularity as data which, convoluted with some approximate identity, generates a family of smooth solutions for the PDE in \eqref{eq:cauchywave} (with more general nonlinearities and data) that does not converge in the space of super-critical Sobolev regularity due to a norm-inflation mechanism; or in \cite{CG23, CG24}, where analogous works are given for super-critical nonlinear Schrödinger equations.
\end{remark}

In Section \ref{sec:prob}, we provide the positive result after a suitable randomization of the initial data, using the work \cite{GOTW22}. In Section \ref{sec:linearcex}, following \cite{Sj70} and \cite{Peral1980}, we prove that there exist explicit elements $F_{\lambda}$ for $x \in \mathbb{R}^d$ which serve as counterexample for the propagation of $L^p$-based regularity under the linear flow, for $1 < p < 2$. In Section \ref{sec:localization}, we give a periodization argument to see the validity of the same non-propagation for $f_{\lambda}$. In Section \ref{sec:dual}, we prove Theorem \ref{thm:nonpropag-linear} using this loss from functions $f_{\lambda}$, as well as a duality argument via the uniform boundedness principle. This allows us to obtain the function $\mathcal{H}$ that eventually appears in \eqref{eq:family-S}. In Section \ref{sec:proof-nonlinearcex}, we extend this phenomenon to the nonlinear case with $d=3$ and $p > 2$ with the proof of Theorem \ref{thm:nonpropag-nonlinear}. Furthemore, in appendices \ref{app:bessel} and \ref{app:temp-dist-xalpha+} we provide some auxiliary results concerning Bessel functions and homogeneous distributions, which will be used throughout this work.

                                                                                                                                                                                                                                                                                                                                                                                                                                                                                                                                                                                                                                                                                                                                                                                                                                                                                                                                                                                                                                                                                                                                                                                                                                                                                                                                                                                                                                                                                   \section{Propagation of almost sure regularity under a Gaussian measure along the flow} \label{sec:prob}
We know that there is global well-posedness of the Cauchy problem \eqref{eq:cauchywave} if the initial datum belongs to $\vec{H}^r(\mathbb{T}^3)$, for $r \geq 1$ (see \cite{GST19}), i.e. it admits a global flow $\Phi_{\text{NLW}}$ on these spaces. In order to be more precise with the notation of the flow, given $t' \in \mathbb{R}$ we write
\begin{align*}
\Phi_{\text{NLW}}(t') :& \quad \vec{H}^r(\mathbb{T}^3) \rightarrow \vec{H}^r(\mathbb{T}^3) \\
&\quad (u_0 , v_0) \mapsto (u(t'),\partial_t u(t')),
\end{align*}
where $(u,\partial_t u)$ is the global solution of the Cauchy problem \eqref{eq:cauchywave} if we substitute the initial datum by $\left. (u,\partial_t u)\right|_{t=0} = (u_0,v_0)$, $u_0 \in H^r(\mathbb{T}^3)$ and $v_0 \in H^{r-1}(\mathbb{T}^3)$.

Given $s \in \mathbb{R}$, let $\vec{\mu}_s$ be the centered Gaussian measure with Cameron-Martin space $\vec{H}^s(\mathbb{T}^3)$. Denote $\vec{u} = (u,v)$. This Gaussian measure has a formal density of the form

\begin{align*}
&d \vec{\mu}_s = Z_s^{-1} e^{-\frac{1}{2} \| \vec{u} \|^2_{\vec{H}^s}} d\vec{u} \nonumber \\
&= \prod_{k \in \mathbb{Z}^3} Z_{s,k}^{-1} e^{-\frac{1}{2} \langle k \rangle^{2s} |\widehat{u}(k)|^2} e^{-\frac{1}{2} \langle k \rangle^{2(s-1)} |\widehat{v}(k)|^2} d \widehat{u}(k) d\widehat{v}(k),
\end{align*}
with $Z_s$ and $Z_{s,k}$ suitable normalization constants.  Samples $\vec{u}^{\omega} = (u^{\omega},v^{\omega})$ from $\vec{\mu}_s$ can be constructed via the Karhunen-Loève expansions
\begin{align}\label{eq:randomid}
u^{\omega}(x) = \sum_{k \in \mathbb{Z}^3} \frac{g_k^{\omega}}{\langle k \rangle^s} e^{ikx} \text{ and }  v^{\omega}(x) = \sum_{k \in \mathbb{Z}^3} \frac{h_k^{\omega}}{\langle k \rangle^{s-1}} e^{ikx},
\end{align}
for $\{g_k^{\omega}\}_{k \in \mathbb{Z}^3}$ and $\{h_k^{\omega}\}_{k \in \mathbb{Z}^3}$ sequences of standard and independent gaussian random variables on $(\Omega,\mathbb{P})$ such that $g_k^{\omega} = \overline{g_{-k}^{\omega}}$ and $h_k^{\omega} = \overline{h_{-k}^{\omega}}$, in order to preserve values in $\mathbb{R}$. One can prove that these series converge in $L^2(\Omega;\vec{H}^r(\mathbb{T}^3))$ as long as
\begin{align*}
r < s - \frac{3}{2},
\end{align*}
and thus the map
\begin{align*}
\omega \in \Omega \mapsto (u^{\omega},v^{\omega})
\end{align*}
induces the Gaussian measure $\vec{\mu}_s$ as a probability measure supported in $\vec{H}^r(\mathbb{T}^3)$ for the same range of $r$. Namely, both $u^{\omega}$ and $v^{\omega}$ converge $\omega$-almost surely in $H^r(\mathbb{T}^3)$ and $H^{r-1}(\mathbb{T}^3)$, respectively. We present now a result on quasi-invariance of this measure under the flow $\Phi_{\text{NLW}}(t)$ from \cite{GOTW22}.

\begin{theorem}\label{thm-qinv}
Let $s \geq 3$ be an odd integer. Then, $\vec{\mu}_s$ is quasi-invariant under the dynamics of the defocusing cubic $NLW$ from \eqref{eq:cauchywave}. Namely, for any $t \in \mathbb{R}$, the Gaussian measure $\vec{\mu}_s$ and its pushforward under $\Phi_{\text{\text{NLW}}}(t)$ are mutually absolutely continuous.
\end{theorem}

This ensures the propagation of almost sure properties with respect to the Gaussian measure along $\Phi_{\text{NLW}}$. In fact, in infinite dimensions this becomes important because many relevant properties regarding the small-scale behaviour under a Gaussian measure hold true with probability $1$ or $0$ (see \cite{Bogachev1998}). This is a consequence of Fernique's theorem (see for instance Theorem 2.7 in \cite{PZ14}): under a Gaussian measure, any given norm is finite with probability $1$ or $0$. 

Denote $\vec{W}^{\sigma,p}(\mathbb{T}^3) = W^{\sigma,p}(\mathbb{T}^3) \times W^{\sigma-1,p}(\mathbb{T}^3)$ for $\sigma \in \mathbb{R}$ and $1 \leq p \leq \infty$. We prove that $\vec{\mu}_s(\vec{W}^{\sigma,p}(\mathbb{T}^3)) = 1$ if $\sigma < s - \frac{3}{2}$ at the end of this section (see Proposition \ref{prop:asprop}), preceded by a useful result which proof can be found in \cite{BurqTzvet}. By the quasi-invariance given in Theorem \ref{thm-qinv},
\begin{align*}
\vec{\mu}_s(\Phi_{\text{NLW}}(t)(\vec{W}^{\sigma,p}(\mathbb{T}^3))) = 1
\end{align*}
for any $t \in \mathbb{R}$, $\sigma < s - \frac{3}{2}$ and $1 \leq p \leq \infty$. 

Consider $p > 2$, $\alpha \in [\max\{\frac{1}{2} - \frac{3}{p} , 0\},1 - \frac{2}{p})$ and $\sigma \in \mathbb{R}$ such that $\sigma + \alpha > \frac{3}{2}$, and fix $s > \sigma + \alpha + \frac{3}{2}$ and $t \in (-1/2,1/2) \setminus \{0\}$. Thus, the regularity in terms of these $\vec{W}^{\sigma + \alpha,p}(\mathbb{T}^3)$ spaces is $\vec{\mu}_s$-almost surely propagated, in contrast to Theorem \ref{thm:nonpropag-nonlinear}, where we proved that the pathological set $\Theta_{\mathcal{H}}$ of data in \eqref{eq:cauchywave} is dense in $W^{\sigma + \alpha,p}(\mathbb{T}^3)$ and gives 
\begin{align*}
\Phi_{\text{NLW}}(t)( \{(0, \langle D \rangle g) : g \in \Theta_{\mathcal{H}}\} ) \cap \vec{W}^{\sigma,p}(\mathbb{T}^3) = \emptyset.
\end{align*}
\begin{lemma}\label{lemma:1deg-wchaos}
		Let $d \geq 1$, $(a_k)_k \in \ell^2(\mathbb{Z}^d)$ and $(g^{\omega}_k)$ a sequence of complex, independent, standard Gaussian random variables . Then, the random variable
		\begin{align*}
			\omega \mapsto \sum_{k \in \mathbb{Z}^d} a_k g^{\omega}_k
		\end{align*}
		is sub-Gaussian (up to normalization to get unit variance). In particular, there exists $C>0$ such that, for every $r \geq 1$,
		\begin{align}
			\| \sum_{k \in \mathbb{Z}^d} a_k g^{\omega}_k\|_{L^r_{\omega}} \leq C \sqrt{r} \|(a_k)\|_{\l^2(\mathbb{Z}^d)}.
		\end{align}
		\label{lemma:randomsum-key}
	\end{lemma}

\begin{proposition}\label{prop:asprop}
Let $s \in \mathbb{R}$, $1 \leq p \leq \infty$ and $\sigma < s - \frac{3}{2}$. Then, $(u^{\omega},v^{\omega}) \in \vec{W}^{\sigma,p}(\mathbb{T}^3)$ $\omega$-almost surely. In other words, $\vec{\mu}_s(\vec{W}^{\sigma,p}(\mathbb{T}^3)) = 1$.
\end{proposition}
\begin{proof}
Let $p \in [1,\infty)$. If $p \leq Q < \infty$, by Minkowski inequality and Lemma \ref{lemma:1deg-wchaos}
\begin{align*}
\| \| u^{\omega} \|_{W^{\sigma,p}(\mathbb{T}^3)} \|_{L^Q_w} \lesssim \| \| \langle D \rangle^{\sigma} u^{\omega} \|_{L^Q_{\omega}} \|_{L^p(\mathbb{T}^3)} \lesssim \sqrt{Q} \| (\langle k \rangle^{\sigma - s})_{k \in \mathbb{Z}^3} \|_{\l^2(\mathbb{Z}^3)},
\end{align*}
which converges if $\sigma < s - \frac{3}{2}$. If $Q < p$, by Hölder's inequality
\begin{align*}
\| \| u^{\omega} \|_{W^{\sigma,p}(\mathbb{T}^3)} \|_{L^Q_{\omega}} \lesssim \| \| \langle D \rangle^{\sigma} u^{\omega} \|_{L^p_{\omega}} \|_{L^p(\mathbb{T}^3)} \lesssim \sqrt{p} \| (\langle k \rangle^{\sigma - s})_{k} \|_{\l^2(\mathbb{Z}^3)},
\end{align*}
which again converges if $\sigma < s - \frac{3}{2}$. Thus, for any $Q \in [1,\infty)$ and $\sigma < s - \frac{3}{2}$,
\begin{align*}
\| \| u^{\omega} \|_{W^{\sigma,p}(\mathbb{T}^3)} \|_{L^Q_{\omega}} \lesssim \sqrt{Q} \left(\max_{Q' \in [1,p]} \sqrt{p/Q'}\right) = \sqrt{Qp}.
\end{align*}
By Markov inequality and proceeding in an identical way with $v^{\omega}$, this implies that there exist some suitable constants $C,c>0$ independent of $\lambda$ such that
\begin{align*}
\vec{\mu}_s(\|(u^{\omega},v^{\omega})\|_{\vec{W}^{\sigma,p}} > \lambda) \leq C e^{-c\lambda^2/p} \text{ for any } \lambda > 0.
\end{align*}
In order to pass to probability $1$, denote
\begin{align*}
A_N = \{\omega \in \Omega : \| (u^{\omega},v^{\omega}) \|_{\vec{W}^{\sigma,p}(\mathbb{T}^3)}> N\}, \quad N \in \mathbb{N}.
\end{align*}
We obtain that
\begin{align*}
\sum_{N \in \mathbb{N}} \mathbb{P}(A_N) \lesssim \sum_{N \in \mathbb{N}} e^{-cN^2/p} < \infty.
\end{align*}
By Borel-Cantelli Lemma, this implies that 
\begin{align*}
\mathbb{P}(\limsup_{N \rightarrow \infty} A_N) = 0 \Longrightarrow \mathbb{P}(\liminf_{N \rightarrow \infty} A_N^c) = 1,
\end{align*}
i.e. for almost every $\omega \in \Omega$ there exists some $N_{\omega} \in \mathbb{N}$ such that
\begin{align*}
\| (u^{\omega},v^{\omega}) \|_{\vec{W}^{\sigma,p}(\mathbb{T}^3)} \leq N_{\omega}.
\end{align*}
This concludes the proof for the case $p < \infty$. 

\hfill \break

Regarding $p=\infty$, with the operator $D = \frac{1}{i}\nabla$ adapted to functions on $\mathbb{T}^3$, given $j \geq 1$ consider the projections $Q_j = \varphi_j(D)$, $Q_0 = \varphi_0(D)$, with $\varphi_0, \varphi_j \in C^{\infty}(\mathbb{R}^3)$ and $\supp(\varphi_0) \subset \overline{B}(0,1)$, $\supp(\varphi_j) \subset \overline{B}(0,2^{j+1})\setminus \overline{B}(0,2^{j-1})$. Since $\sigma < s - \frac{3}{2}$, we know that there exists some $q \in [1,\infty)$ such that $\sigma < s - \frac{3}{2} - \frac{1}{q}$. Consider $Q \geq 1$, so that by Bernstein inequality and a similar procedure to the case $p < \infty$,
\begin{align*}
&\| \| Q_j u^{\omega} \|_{W^{\sigma,\infty}} \|_{L^Q_{\omega}} \lesssim 2^{\frac{j}{q}} \| \| Q_j u^{\omega} \|_{W^{\sigma,q}} \|_{L^Q_{\omega}} \lesssim \sqrt{Qq} 2^{j\left( \frac{1}{q} + \sigma - s + \frac{3}{2} \right)}, \nonumber \\
&\| \| Q_0 u^{\omega} \|_{W^{\sigma,\infty}} \|_{L^Q_{\omega}} \lesssim \sqrt{Qq}.
\end{align*}
Let $(\sigma_j)_{j \in \mathbb{N} \cup \{0\}} \subset (0,\infty)$ be a sequence in the unit ball of $\l^1(\mathbb{N})$ such that\footnote{It would suffice to take, for each $j \in \mathbb{N} \cup \{0\}$,
\begin{align*}
\sigma_j := c_{\sigma}^{-1} 2^{-\frac{j}{2}(s -  \sigma - \frac{3}{2} - \frac{1}{q})}, \quad c_{\sigma} = \sum_{j \in \mathbb{N} \cup \{0\}} 2^{-\frac{j}{2}(s - \sigma - \frac{3}{2} - \frac{1}{q})}.
\end{align*}}
\begin{align*}
\sigma_j \geq 2^{-j(s -  \sigma - \frac{3}{2} - \frac{1}{q})}, \quad j \geq j_0,
\end{align*}
for some $j_0 \in \mathbb{N} \cup \{0\}$. This choice of $\sigma_j$ let us take $C,c > 0$ independent of $j$ such that (doing also the analogous computation for $v^{\omega}$)
\begin{align*}
&\vec{\mu}_s(\| (Q_j u^{\omega},Q_jv^{\omega}) \|_{\vec{W}^{\sigma,\infty}(\mathbb{T}^3)} > \sigma_j \lambda) \leq C e^{-\frac{c\lambda^2 \sigma_j^2}{q} 2^{j(2s - 2 \sigma - 3 - \frac{2}{q})}}, \nonumber \\
&\vec{\mu}_s(\| (Q_0 u^{\omega},Q_0 v^{\omega}) \|_{\vec{W}^{\sigma,\infty}(\mathbb{T}^3)} > \sigma_0 \lambda) \leq C e^{-\frac{c\lambda^2\sigma_0^2}{q}}.
\end{align*}
Summing all the pieces,
\begin{align*}
&\vec{\mu}_s(\| (u^{\omega},v^{\omega}) \|_{\vec{W}^{\sigma,\infty}} > \lambda) \nonumber \\ 
&\leq \vec{\mu}_s(\| (Q_0 u^{\omega},Q_0v^{\omega}) \|_{\vec{W}^{\sigma,\infty}} > \sigma_0 \lambda) + \sum_{j \in \mathbb{N}} \mu_s(\| (Q_j u^{\omega},Q_jv^{\omega}) \|_{\vec{W}^{\sigma,\infty}} > \sigma_j \lambda) \\
&\lesssim \sum_{j \in \mathbb{N} \cup \{0\}} e^{-\frac{c\lambda^2\sigma_j^2}{q} 2^{j(2s - 2 \sigma - 3 - \frac{2}{q})}} \lesssim e^{-c_1\lambda^2/q}
\end{align*}
for some constant $c_1 > 0$ independent of $\lambda$. Therefore, using Borel-Cantelli Lemma in the same way as above, we have that 
\begin{align*}
(u^{\omega},v^{\omega}) \in \vec{W}^{\sigma,\infty}(\mathbb{T}^3), \quad \omega\text{-almost surely}.
\end{align*}
\end{proof}

\section{Deterministic counterexample on $\mathbb{R}^d$}\label{sec:linearcex}	 
Let $\mathcal{F} : \mathcal{S}'(\mathbb{R}^d) \rightarrow \mathcal{S}'(\mathbb{R}^d)$ the distributional Fourier transform map. Let $\psi$ be a symmetric mollifier, i.e. an element of $C^{\infty}_c(\mathbb{R}^d)$ such that
\begin{align*}
&\bullet \int_{\mathbb{R}^d}\psi(x)dx = 1, \\
&\bullet \lim_{\varepsilon \rightarrow 0^+} \varepsilon^{-d} \psi(\cdot/\varepsilon) = \delta \text{ in } \mathcal{S}'(\mathbb{R}^d), \\
&\bullet \psi(\xi) = \psi_1(|\xi|) \text{ for any } \xi \in \mathbb{R}^d,
\end{align*}
where $\psi_1 : [0,\infty) \rightarrow \mathbb{R}$ is smooth (chosen to have support contained in $[0,1]$) and $\delta$ is the Dirac delta distribution. Denote, for any $\xi \in \mathbb{R}^d$, $\psi_{1/4}(\xi) = 4^d \psi(4 \xi)$ and
\begin{align}\label{eq:def-h}
h(\xi) = 1 - (\psi_{1/4} \ast \mathrm1_{B(0,3/4)})(\xi),
\end{align}
where $\mathrm1_{B(0,3/4)}$ is the characteristic function of the ball centered at $0$ with radius $3/4$ in $\mathbb{R}^d$. Therefore, $h$ is a smooth function on $\mathbb{R}^d$ which is $0$ on $B(0,1/2)$ (as well as all its derivatives), and which is $1$ on $B(0,1)^c$. This setting allows us to define the main objects of this section. Given $\lambda \in (0,d)$, define the following elements of $\mathcal{S}'(\mathbb{R}^d)$:
\begin{align}\label{def:g-lambda}
F_{\lambda} = \mathcal{F}^{-1}[g_{\lambda}], \quad \text{where }  g_{\lambda}(\xi) =  h(\xi)|\xi|^{-\lambda}.
\end{align}

\begin{proposition}\label{prop:Flambda-inWsigmap}
Let $p\in[1,\infty)$. The elements $F_{\lambda}$ given in \eqref{def:g-lambda} belong to $L^p(\mathbb{R}^d)$ if $\lambda \in (d - \frac{d}{p},d)$.
\end{proposition}

\begin{theorem}\label{thm:nonpropag-linear-rn}
Let $1 < p < 2$, $0 \leq \alpha < (d-1)\left(\frac{1}{p} - \frac{1}{2} \right)$ and $\lambda \in (d - \frac{d}{p} , d)$ such that $\lambda + \alpha \in ( \frac{d-1}{2},d)$, $\frac{d-1}{2} - \lambda -  \alpha \notin -\mathbb{N}$ and $\lambda + \alpha \leq \frac{d+1}{2} - \frac{1}{p}$. Then, for any $t \neq 0$,
\begin{align*}
F_{\lambda} \in L^p(\mathbb{R}^d) \quad \text{and} \quad |D|^{-\alpha} \sin(t|D|) F_{\lambda} \notin L^p(\mathbb{R}^d).
\end{align*}
\end{theorem}

\begin{remark}
Let $1 < p < 2$ and $\alpha \geq 0$. The condition $\alpha < (d-1) \left( \frac{1}{p} - \frac{1}{2} \right)$ from Theorem \ref{thm:nonpropag-linear-rn} is necessary due to the compatibility of the conditions
\begin{align}\label{eq:req-cond}
&\lambda > d - \frac{d}{p}, \qquad \text{needed for } F_{\lambda} \in L^p(\mathbb{R}^d), \nonumber \\
&\lambda + \alpha \leq \frac{d+1}{2} - \frac{1}{p}, \qquad \text{needed for } |D|^{-\alpha} \sin(t|D|) F_{\lambda} \notin L^p(\mathbb{R}^d).
\end{align}
In particular, both inequalities imply that
\begin{align*}
&\left( d - \frac{d}{p} \right) + \alpha < \lambda + \alpha \leq \frac{d+1}{2} - \frac{1}{p} \\
&\qquad \qquad \Longrightarrow \alpha < \left( \frac{d+1}{2} - \frac{1}{p} \right) - \left( d - \frac{d}{p} \right) = (d-1) \left( \frac{1}{p} - \frac{1}{2} \right).
\end{align*}
\end{remark}

\begin{remark}
Let $1 < p < 2$· and $0 \leq \alpha < (d-1)\left(\frac{1}{p} - \frac{1}{2} \right)$. Regarding the numerology in Theorem \ref{thm:nonpropag-linear-rn}, which appears again in Theorem \ref{thm:divergence-torus} and in the proof of Theorem \ref{thm:nonpropag-linear}, we can establish the existence of at least one $\lambda \in (d - \frac{d}{p} , d)$ such that
\begin{align*}
\lambda + \alpha \in \left( \frac{d-1}{2},d \right) \quad \text{and} \quad \lambda + \alpha \leq \frac{d+1}{2} - \frac{1}{p}.
\end{align*}
In other words, we need to prove that the interval $I = (G_1(\alpha),G_2(\alpha))$ is nonempty, where
\begin{align*}
&G_1(\alpha) = \max\left\{\frac{d-1}{2} - \alpha , d - \frac{d}{p} \right\}, \\
& G_2(\alpha) = \min\left\{ \frac{d+1}{2} - \frac{1}{p} -  \alpha , d - \alpha ,d  \right\} = \frac{d+1}{2} - \frac{1}{p} -  \alpha.
\end{align*}
Notice that, if $p < \frac{2d}{d+1}$,
\begin{align*}
&G_1(\alpha) = \left\{ \begin{array}{lc} \frac{d-1}{2} - \alpha, & \text{for } \alpha \in [0 , \frac{d}{p} - \frac{d+1}{2}), \\ \\ d - \frac{d}{p}, & \text{for } \alpha \in [\frac{d}{p} - \frac{d+1}{2} , \frac{d-1}{p} - \frac{d-1}{2}). \end{array} \right. \\ 
\end{align*}
If $p \geq \frac{2d}{d+1}$, then $G_1(\alpha) = d - \frac{d}{p}$ in the whole interval. Thus, distinguishing between the possible values that $G_1$ can take, $I \neq \emptyset$ follows from
\begin{itemize}
\item $\frac{d-1}{2} < \frac{d+1}{2} - \frac{1}{p}$ if and only if $p > 1$,
\item and $d - \frac{d}{p} < G_2(\alpha)$ if and only if $\alpha < \frac{d-1}{p} - \frac{d-1}{2}$.
\end{itemize}
\end{remark}

\begin{remark}\label{rk:notpmore2}
The numerology from Theorem \ref{thm:nonpropag-linear-rn} cannot be adapted to $2 < p < \infty$. We can argue as follows. Let $0 \leq \alpha < (d-1)\left( \frac{1}{2} - \frac{1}{p} \right)$. Then, the conditions $\lambda > d - \frac{d}{p}$ and $\lambda + \alpha \leq \frac{d+1}{2} - \frac{1}{p}$ are not compatible, meaning that both conditions in \eqref{eq:req-cond} cannot be satisfied simultaneously. Indeed, if we consider the smallest possible $\alpha$ ($\alpha = 0$), we have that
\begin{align*}
d - \frac{d}{p} < \frac{d+1}{2} - \frac{1}{p} - \alpha \Longleftrightarrow p < 2.
\end{align*}
In fact, if $p>2$ we can apply Hausdorff-Young inequality in the following way:
\begin{align*}
\| S(t)(0,\langle D \rangle F_{\lambda})\|_{L^p(\mathbb{R}^d)} \lesssim \| \sin(t|\xi|) h(\xi) |\xi|^{-\lambda} \|_{L^{p'}(\mathbb{R}^d)},
\end{align*}
where the last integral converges if $\lambda > d - \frac{d}{p}$. This is the same condition to have $F_{\lambda} \in L^p(\mathbb{R}^d)$, meaning that the functions $F_{\lambda}$ do not serve to give an explicit counterexample for the loss of $L^p$-based regularity for $p > 2$. The analogous situation takes place for the periodic case, regarding Theorem \ref{thm:divergence-torus} below.
\end{remark}

Proposition \ref{prop:Flambda-inWsigmap} will be a straightforward consequence of Lemma \ref{lemma:Flambda} below, where we characterize the functions $F_{\lambda}$. In order to prove Theorem \ref{thm:nonpropag-linear-rn}, we provide some preliminary results that will be useful later to analyze $| D |^{- \alpha} e^{it|D|} F_{\lambda}$. Namely, Lemma \ref{lemma:ftransformradial_bessel} below will let us use an integral representation of radial elements in $L^p(\mathbb{R}^d)$, $1 \leq p < \infty$, in terms of Bessel functions, and Lemma \ref{lemma:invft-xalpha-nu} gives us the explicit form for the distributional inverse Fourier transforms of certain homogeneous distributions that will come later (its proof will follow the strategy from \cite{GS64}). For the sake of completeness, we include their proofs in Appendices \ref{app:bessel} and \ref{app:temp-dist-xalpha+}.

\begin{lemma}\label{lemma:Flambda}
Let $d \in \mathbb{N}$, $x \in \mathbb{R}^d \setminus \{0\}$ and $\lambda \in (0,d)$. Then
\begin{align*}
F_{\lambda}(x) = \gamma(d,\lambda) |x|^{-d+\lambda} + G_{\lambda,d}(x),
\end{align*}
where $G_{\lambda,d} \in C^{\infty}(\mathbb{R}^d)$ and $\gamma(d,\lambda)$ is a constant only depending on $d$ and $\lambda$. On the other hand,
\begin{align*}
|F_{\lambda}(x)| \lesssim_{d,\lambda} |x|^{-N}
\end{align*} 
for any $N > d - \lambda$.
\end{lemma}
\begin{proof}
Let $x \neq 0$. Let $\varphi \in C^{\infty}_c(\mathbb{R}^d\setminus \{0\})$ with $\supp \varphi \subset \{\xi \in \mathbb{R}^d : \frac{1}{2} \leq |\xi| \leq 2\}$ satisfying
\begin{align*}
\sum_{j = -\infty}^{\infty} \varphi(2^{-j}\xi) = 1, \quad \xi \neq 0,
\end{align*}
so that for any $\xi \neq 0$ there are at most two terms in this sum which are nonzero (see page $197$ of \cite{MS1} for details). This allows us to consider the following Littlewood-Paley decomposition. Let $P_j = \varphi(D/2^j)$ and $B_j = \overline{B}(0,2^{j+1}) \setminus \overline{B}(0,2^{j-1})$ for any $j \in \mathbb{Z}$. Given that $g_{\lambda}(\xi) = 0$ if $|\xi| \leq \frac{1}{2}$, we have that
\begin{align*}
| F_{\lambda}(x)| \lesssim \sum_{j \geq -1} \left| \int_{B_j} \varphi(\xi/2^j) g_{\lambda}(\xi) e^{i\xi x} d\xi \right| = \sum_{j\geq -1} |I_j|.
\end{align*}
By doing a change of variables $\xi = 2^j \eta$, we obtain
\begin{align}\label{eq:oscint-0}
&I_j = 2^{j(d - \lambda)} \int_{B_0} h(2^j \eta) \varphi(\eta) |\eta|^{-\lambda} e^{i(2^j |x|) \eta \cdot \frac{x}{|x|}} d\eta, \quad j \geq -1.
\end{align}
Let $j \geq 1$. Then \eqref{eq:oscint-0} is an oscillatory integral with a phase $(2^j |x|)\phi(\eta)$ for which $\nabla_{\eta} \phi(\eta) = x/|x|$, which is not zero because $x \neq 0$ and which satisfies $|\nabla_{\eta}\phi(\eta)| = 1$ for any $\eta \in B_0$. Namely
\begin{align}\label{eq:oscint-o1}
|I_j| = 2^{j(d - \lambda)} \left| \int_{B_0} \varphi(\eta) |\eta|^{-\lambda} e^{i(2^j |x|) \eta \cdot \frac{x}{|x|}} d\eta \right|,
\end{align}
where $\varphi(\eta)|\eta|^{-\lambda}$ is smooth and compactly supported on $B_0$. Then, for any $N \in \mathbb{N}$,
\begin{align*}
|I_j| \lesssim_{d,\lambda,\varphi} |x|^{-N} 2^{j(d - \lambda - N)}.
\end{align*}
Regarding $j \in \{-1,0\}$, notice that $\varphi(\eta) |\eta|^{-\lambda} h(2^j \eta)$ is also smooth and compactly supported on $B_0$, so we can obtain \eqref{eq:oscint-o1} too. All in all,
\begin{align}\label{eq:nonstatphase-oscint}
|F_{\lambda}(x)| \lesssim |x|^{-N} \left( \sum_{j \geq -1} 2^{j(d-\lambda-N)} \right),
\end{align}
which requires $N > d - \lambda$ to be convergent at the right-hand side. \\

On the other hand, we have
\begin{align*}
g_{\lambda}(\xi) = |\xi|^{-\lambda} +   |\xi|^{-\lambda} ( h(\xi) - 1 )
\end{align*}
in $\mathcal{S}'(\mathbb{R}^d)$ (recall that $\lambda \in (0,d)$), so
\begin{align}\label{eq:Flambda-proof1}
 F_{\lambda} = \mathcal{F}^{-1}[|\xi|^{-\lambda}] + G_{\lambda,d},
\end{align}
where
\begin{align*}
G_{\lambda,d}(x) = \int_{B(0,1)} e^{i\xi x} |\xi|^{-\lambda}(h(\xi) - 1)d\xi.
\end{align*}
Then, by the dominated convergence theorem, $G_{\lambda,d}$ is infinitely differentiable as long as $\lambda < d$. Regarding the first term, from standard theory\footnote{In fact, if we take Gaussian test functions and perform standard computations, we obtain
\begin{align*}
\gamma(d,\lambda) = \pi^{\lambda-\frac{d}{2}}\frac{\Gamma(\frac{d-\lambda}{2})}{\Gamma(\frac{\lambda}{2})}
\end{align*}} on homogeneous distributions one can prove that $\mathcal{F}[|\xi|^{-\lambda}] = \gamma(d,\lambda) |x|^{\lambda-d}$ in $\mathcal{S}'(\mathbb{R}^d)$ (see \cite{wolff} for  details), where $\gamma(d,\lambda)$ is some constant.
Therefore,
\begin{align*}
F_{\lambda}(x) = \gamma(d,\lambda) |x|^{-d+\lambda} + G_{\lambda,d}(x)
\end{align*}
in $\mathcal{S}'(\mathbb{R}^d)$. This concludes the proof.
\end{proof}

\begin{lemma}\label{lemma:ftransformradial_bessel}
		Let $p \in [1,\infty)$. Suppose $f$ is a radial function in  $L^p(\mathbb{R}^d)$, $d \geq 2$; thus, $f(x) = f_0(|x|)$ for a.e. $x \in \mathbb{R}^d$. Moreover, assume that the integral
\begin{align}\label{eq:convcond-radialexp}
\int_0^{\infty} f_0(\rho) J_{\frac{d-2}{2}}(\rho|\xi|) \rho^{d/2} d\rho
\end{align}		
converges in a.e. $\xi \in \mathbb{R}^d$. Then the Fourier transform $\mathcal{F}[f]$ is given also by a radial function, and it has the form
		\begin{align*}
			\mathcal{F}[f](\xi) = c(d) |\xi|^{-\frac{d-2}{2}} \int_0^{\infty} f_0(\rho) J_{\frac{d-2}{2}}(\rho|\xi|) \rho^{d/2} d\rho \quad \text{for a.e. } \xi \in \mathbb{R}^d,
		\end{align*}
		where $c(d)$ is some constant depending only on the dimension.
	\end{lemma}
	\begin{proof}
	See Appendix \ref{app:bessel}.
	\end{proof}
	We denote $x_+$ and $x_-$ the functions $x \mathbb{I}_{(0,\infty)}(x)$ and $-x \mathbb{I}_{(-\infty,0)}(x)$ defined on $\mathbb{R}$. The proof of the following Lemma is given in Appendix \ref{app:temp-dist-xalpha+}.
	
	\begin{lemma}\label{lemma:invft-xalpha-nu}
	Let $z \in \mathbb{C}$ such that $\text{Re}(z) \notin -\mathbb{N}$. Then, the tempered distributions $x^z \mathbb{I}_{(1,\infty)}(x)$ admit the inverse distributional Fourier transform
	\begin{align*}
	&\mathcal{F}^{-1}[x^z \mathbb{I}_{(1,\infty)}(x)](\theta)= J(z) \{ e^{i\frac{\pi}{2}(z+1)} \theta_+^{-z-1} + e^{-i\frac{\pi}{2}(z+1)} \theta_-^{-z-1} \},
	\end{align*}
	where $J : \mathbb{C} \rightarrow \mathbb{C}$ is a function such that
	\begin{align*}
	|J(z)| \leq \Gamma(\text{Re}(z)+1) + \frac{\pi}{2}.
	\end{align*}
	\end{lemma}
	
Given $t \neq 0$, the following result provides us with the explicit form of $| D |^{- \alpha} e^{it|D|} F_{\lambda}$, which will let us identify the conditions on $\lambda$, $\alpha$ and $p$ for the non-convergence of 
\begin{align*}
\Big\| | D |^{- \alpha} \sin(t|D|) F_{\lambda} \Big\|_{L^p(\mathbb{R}^d)}.
\end{align*}

\begin{theorem}\label{thm:flambda-tnotzero}
Let $t \in \mathbb{R} \setminus \{0\}$, $\alpha \geq 0$ and $\lambda \in \mathbb{R}$ such that $\lambda + \alpha \in (\frac{d-1}{2},d)$ and $\frac{d-1}{2} - \lambda - \alpha \notin -\mathbb{N}$, and $J$ the function given in Lemma \ref{lemma:invft-xalpha-nu}. Then the following holds:
\begin{itemize}
\item For a.e. $x \in \mathbb{R}^d \setminus \{0\}$ such that $|x| \neq t$ and $-|x| \neq t$:
\begin{align*}
| &D |^{- \alpha} e^{it|D|} F_{\lambda}(x) \nonumber \\
&= J\left(\frac{d-1}{2}-\lambda - \alpha\right) c(d) |x|^{-\frac{d-1}{2}} \left\{ c_+ [ e^{i\frac{\pi}{2}(\frac{d+1}{2} - \lambda - \alpha)} (|x| + t)_+^{\lambda + \alpha -\frac{d+1}{2}}\right. \\
&\hspace{53mm} + e^{-i\frac{\pi}{2}(\frac{d+1}{2} - \lambda - \alpha)} \left(|x| + t\right)_-^{\lambda+\alpha-\frac{d+1}{2}}] \nonumber \\
&\hspace{48mm} + c_-  [ e^{i\frac{\pi}{2}(\frac{d+1}{2} - \lambda - \alpha)}(-|x| + t)_+^{\lambda+\alpha-\frac{d+1}{2}} \nonumber \\
&\hspace{53mm} \left. + e^{-i\frac{\pi}{2}(\frac{d+1}{2} - \lambda - \alpha)} \left( -|x| + t\right)_-^{\lambda + \alpha-\frac{d+1}{2}} ] \right\} \nonumber  \\
&\hspace{48mm} + c_{\lambda}(x) + D_{\lambda}(t,x),
\end{align*}
for $c_+$ and $c_-$ the constants defined in \eqref{eq:cplusminus}, $D_{\lambda}(t,\cdot) \in C^{\infty}(\mathbb{R}^d)$ and 
\begin{align}\label{eq:clambda}
|c_{\lambda}(x)| \lesssim |x|^{-\frac{d+1}{2}}.
\end{align}

\item  Moreover, for any $A > 1$, $x \in \mathbb{R}^d \setminus B(0,A|t|)$ and $N > d - \lambda - \alpha$, we have that
\begin{align*}
|| D |^{- \alpha} e^{it|D|} F_{\lambda}(x)| \lesssim_{d,\lambda,\alpha} |x|^{-N}.
\end{align*}

\end{itemize}
\end{theorem}

\begin{proof}
Let $A > 1$ and $x \in \overline{B}(0,A|t|)^c$. We can repeat the procedure from the Littlewood-Paley decomposition in the proof of Lemma \ref{lemma:Flambda} with a slight modification. Indeed, maintaining the same notation,
\begin{align*}
&|| D |^{- \alpha} e^{it|D|} F_{\lambda}(x)| \lesssim \sum_{j \geq -1} \left| \int_{B_j} e^{i \xi x + t |\xi|}(\varphi(\xi/2^j) |\xi|^{-\lambda - \alpha} h(\xi)) d\xi \right| \nonumber \\
&= \sum_{j \geq 1} 2^{j(d-\lambda-\alpha)} \left| \int_{B_0} e^{i 2^j |x| (\frac{\eta x + t |\eta|}{|x|})}(\varphi(\eta) |\eta|^{-\lambda-\alpha} ) d\eta \right| \nonumber \\
&+ \sum_{j=-1}^0 2^{j(d-\lambda- \alpha)} \left| \int_{B_0} e^{i 2^j |x| (\frac{\eta x + t |\eta|}{|x|})}(\varphi(\eta) |\eta|^{-\lambda-\alpha} h(2^j \eta)) d\eta \right|.
\end{align*}
The phase $\phi(\xi) = \frac{\eta x + t |\eta|}{|x|}$ does not have critical points for $|x| > |t|$. Indeed,
\begin{align*}
\nabla_{\eta} \phi(\eta) = \frac{1}{|x|} \left( x + \frac{t \eta}{|\eta|} \right),
\end{align*}
so $\nabla_{\eta} \phi(\eta) = 0$ would imply $|x| = |t|$. Moreover, since $|x| > A |t|$, we have that for any $\eta \in B_0$
\begin{align*}
|\nabla_{\eta} \phi(\eta)| \geq 1 - \frac{|t|}{|x|} > 1 - \frac{1}{A}.
\end{align*}
Moreover, $\varphi(\eta) |\eta|^{-\lambda} h(2^j \eta)$ is smooth and compactly supported on $B_0$ for any $j \geq -1$ (notice that $h(2^j \cdot) = 1$ for $j \geq 1$).  Thus we can do the standard integration by parts to obtain
\begin{align*}
|| D |^{- \alpha} e^{it|D|} F_{\lambda}(x)| \lesssim_{d,\lambda,\alpha,\varphi} |x|^{-N} \left( \sum_{j \geq -1} 2^{j(d-\lambda-\alpha-N)} \right),
\end{align*}
where $N > d - \lambda-\alpha$ suffices for the convergence of the right-hand side. 

Let $x \in \mathbb{R}^d \setminus \{0\}$ such that $|x| \neq t$ and $-|x| \neq t$. Since $\lambda + \alpha < d$, there exists some $q > 1$ such that $\lambda + \alpha > \frac{d}{q}$. Recalling \eqref{def:g-lambda}, this implies that 
\begin{align*}
|\xi|^{-\alpha} e^{it|\xi|} g_{\lambda}(\xi) \in L^q(\mathbb{R}^d).
\end{align*}
We notice that, as a particular consequence of the upcoming computations, the integral
\begin{align*}
\int_1^{\infty} \frac{e^{it\rho}}{\rho^{\lambda + \alpha}} J_{\frac{d-2}{2}}(\rho|x|) \rho^{d/2} d\rho
\end{align*}
converges for almost every $x \in \mathbb{R}^d$. Thus, Lemma \ref{lemma:ftransformradial_bessel} applied to $f = \mathrm{1}_{|\xi| > 1} |\xi|^{-\alpha} e^{it|\xi|} g_{\lambda}(\xi)$ implies that
\begin{align*}
&|D|^{-\alpha} e^{it|D|} F_{\lambda}(x) \nonumber \\
&=c(d) |x|^{-\frac{d-2}{2}} \int_1^{\infty}  e^{it\rho}  \rho^{-\lambda-\alpha} J_{\frac{d-2}{2}}(\rho|x|) \rho^{d/2} d\rho + \int_{|\xi| \leq 1} e^{ix\xi + it|\xi|} |\xi|^{-\lambda-\alpha} h(\xi) d\xi \nonumber \\
&=(i) + (ii), \qquad \text{for a.e. } x \in \mathbb{R}^d \setminus \{0\}. 
\end{align*}
The object $(ii)$ is in $C^{\infty}(\mathbb{R}^d)$ for each $t \neq 0$ and $\lambda+\alpha < d$. We focus on $(i)$. From Bessel functions asymptotic expansion \eqref{eq:asympBessel} we have that
\begin{align*}
J_{\frac{d-2}{2}}(s)  - \left( c_+ s^{-1/2} e^{is} + c_- s^{-1/2} e^{-is} \right) = R_{\frac{d-2}{2}}(s),
\end{align*}
for $R_{\frac{d-2}{2}}(s) = \mathcal{O}(s^{-3/2})$ when $s \rightarrow \infty$, and 
$c_+, c_-$ the constants defined in \eqref{eq:cplusminus}. This expansion and the Poisson representation of Bessel functions given in Lemma \ref{lemma:bessel_kreal} implies that, for any $s \in \mathbb{R} \setminus \{0\}$,
\begin{align}\label{eq:ourbessel-asymp}
|s^{\frac{1}{2}}R_{\frac{d-2}{2}}(s)| \lesssim \min\{ 1 , s^{-1} \}.
\end{align}
Observe that $\frac{d-2}{2} > -\frac{1}{2}$ if and only if $d > 1$, i.e. $d>1$ is required for a correct definition of these Bessel functions as we have defined them. Then, up to the constant $c(d)$, $(i)$ is equal to
\begin{align*}
&\quad |x|^{-\frac{d-2}{2}} \left( c_+ \int_1^{\infty}  e^{i\rho|x| + it\rho} \rho^{-\lambda-\alpha} (\rho|x|)^{-1/2} \rho^{d/2} d\rho \right.   \nonumber \\
&\left. \hspace{45mm} + c_- \int_1^{\infty} e^{-i\rho|x| + it\rho} \rho^{-\lambda-\alpha} (\rho|x|)^{-1/2} \rho^{d/2} d\rho \right) \nonumber \\
& \hspace{45mm} + |x|^{-\frac{d-2}{2}} \int_1^{\infty} e^{it\rho} \rho^{-\lambda-\alpha} R_{\frac{d-2}{2}}(\rho|x|) \rho^{d/2} d\rho \nonumber \\
&\quad= |x|^{-\frac{d-1}{2}} \left(c_+ \int_{1}^{\infty}  e^{i \rho ( t + |x|)}  \rho^{\frac{d-1}{2} - \lambda-\alpha} d\rho + c_- \int_{1}^{\infty}  e^{i \rho( t - |x|)}  \rho^{\frac{d-1}{2} - \lambda-\alpha} d\rho \right) \nonumber \\
&\hspace{45mm} + |x|^{-\frac{d-2}{2}} \int_{1}^{\infty} e^{it\rho} R_{\frac{d-2}{2}}(\rho |x|) \rho^{\frac{d}{2} -\lambda-\alpha} d\rho \nonumber \\  
& \quad = |x|^{-\frac{d-1}{2}}(*) + |x|^{-\frac{d-2}{2}}(**).
\end{align*}
Assume that $\frac{d-1}{2} - \lambda - \alpha \notin -\mathbb{N}$. Thus, for $(*)$ we can use Lemma \ref{lemma:invft-xalpha-nu} with $z = \frac{d-1}{2} - \lambda - \alpha$, and $\theta = t + |x|$ for the first term and $\theta = t - |x|$ for the second one:
\begin{align*}
(*) &= J\left(\frac{d-1}{2}-\lambda-\alpha\right) \left\{c_+ [ e^{i\frac{\pi}{2}(\frac{d+1}{2} - \lambda-\alpha)} (|x| + t)_+^{\lambda+\alpha-\frac{d+1}{2}}\right. \\
&\hspace{48mm} + e^{-i\frac{\pi}{2}(\frac{d+1}{2} - \lambda-\alpha)} \left(|x| + t\right)_-^{\lambda+\alpha-\frac{d+1}{2}}] \nonumber \\
&\hspace{48mm} + c_-  [ e^{i\frac{\pi}{2}(\frac{d+1}{2} - \lambda-\alpha)}(-|x| + t)_+^{\lambda + \alpha-\frac{d+1}{2}} \\
&\left. \hspace{48mm}+ e^{-i\frac{\pi}{2}(\frac{d+1}{2} - \lambda - \alpha)} \left( -|x| + t\right)_-^{\lambda + \alpha-\frac{d+1}{2}} ] \right\}.
\end{align*}
With respect to $(**)$, by inequality \eqref{eq:ourbessel-asymp} we obtain that
\begin{align*}
&|x|^{-\frac{d-2}{2}}\left|  (**) \right| \lesssim |x|^{-\frac{d-1}{2}} \int_{1}^{\infty}\rho^{\frac{d-1}{2} - \lambda - \alpha} \min\{ 1 , (\rho|x|)^{-1} \} d\rho \nonumber.
\end{align*}
If we bound the integrand by $\min\{ 1 , (\rho|x|)^{-1} \} \leq (\rho|x|)^{-1}$ we obtain
\begin{align*}
|x|^{-\frac{d-2}{2}}\left|  (**) \right| \lesssim \left( \int_{1}^{\infty} \rho^{\frac{d-3}{2} - \lambda - \alpha} d\rho \right) |x|^{-\frac{d+1}{2}},
\end{align*}
where the integral in the right-hand side converges if $\lambda + \alpha > \frac{d-1}{2}$. Denote $|x|^{-\frac{d-2}{2}} (**) = c_{\lambda}(x)$. Then, if $\lambda + \alpha \in (\frac{d-1}{2},d)$, 
\begin{align*}
|c_{\lambda}(x)| \lesssim |x|^{-\frac{d+1}{2}}.
\end{align*}
This concludes the proof.
\end{proof}

\begin{proof}[Proof of Theorem \ref{thm:nonpropag-linear-rn}]
Recall that
$$|D|^{-\alpha} \sin(t|D|) F_{\lambda} = \frac{1}{2i} |D|^{-\alpha} ( e^{it|D|} -  e^{-it|D|} ) F_{\lambda}.$$
Let $A \in (1,2)$. We focus on the part of the integration domain concerning $(A-1)|t| \leq |x| < A|t|$, which includes $|x| = |t|$. By Theorem \ref{thm:flambda-tnotzero} the quantity
\begin{align*}
\| |D|^{-\alpha} \sin(t|D|) F_{\lambda} \|_{L^p(\mathbb{R}^d)}
\end{align*}
is bounded below, up to constants, by
\begin{align*}
\left\| \right. & \left. |x|^{-\frac{d-1}{2}} \left( c_+ [ e^{i\frac{\pi}{2}(\frac{d+1}{2} - \lambda - \alpha)} (|x| + t)_+^{\lambda + \alpha -\frac{d+1}{2}} + e^{-i\frac{\pi}{2}(\frac{d+1}{2} - \lambda - \alpha)} \left(|x| + t\right)_-^{\lambda + \alpha -\frac{d+1}{2}}] \nonumber \right. \right. \\
&\left. + c_-  [ e^{i\frac{\pi}{2}(\frac{d+1}{2} - \lambda - \alpha)}(-|x| + t)_+^{\lambda + \alpha -\frac{d+1}{2}} + e^{-i\frac{\pi}{2}(\frac{d+1}{2} - \lambda - \alpha)} \left(-|x| + t\right)_-^{\lambda + \alpha-\frac{d+1}{2}} ]\right. \nonumber \\
&\left. - c_+  [ e^{i\frac{\pi}{2}(\frac{d+1}{2} - \lambda - \alpha)}(|x| - t)_+^{\lambda + \alpha-\frac{d+1}{2}} + e^{-i\frac{\pi}{2}(\frac{d+1}{2} - \lambda - \alpha)} \left(|x| - t\right)_-^{\lambda + \alpha -\frac{d+1}{2}} ]\right. \nonumber \\
& - c_-  [ e^{i\frac{\pi}{2}(\frac{d+1}{2} - \lambda - \alpha)}(-|x| - t)_+^{\lambda + \alpha-\frac{d+1}{2}}  \nonumber \\
&\left. \left. \qquad \qquad \qquad \qquad + e^{-i\frac{\pi}{2}(\frac{d+1}{2} - \lambda - \alpha)} \left(-|x| - t\right)_-^{\lambda + \alpha-\frac{d+1}{2}} ]\right) \right\|_{L^p(B(0,A|t|) \setminus B(0,(A-1)|t|))} \nonumber \\
&\qquad - \left\| \Tilde{c}_{\lambda} + \Tilde{D}_{\lambda}(t,\cdot) \right\|_{L^p(B(0,A|t|) \setminus B(0,(A-1)|t|))},
\end{align*}
where $\tilde{D}_{\lambda}(t,\cdot)$ is the sum of two elements of $C^{\infty}(\mathbb{R}^d)$, so in particular $\Tilde{D}_{\lambda}(t,\cdot) \in L^p_{\text{loc}}(\mathbb{R}^d)$; and $\Tilde{c}_{\lambda}$ satisfies \eqref{eq:clambda}, given that $\lambda + \alpha \in (\frac{d-1}{2},d)$. Recall from \eqref{eq:cplusminus} that $c_+ = \overline{c_-}$, and notice that
\begin{align}\label{eq:cases+-}
(|x| + t)_{\pm} = (-|x| - t)_{\mp}, \qquad (-|x| + t)_{\pm} = (|x| - t)_{\mp}.
\end{align}
Thus
\begin{align}\label{eq:gralint-p-2}
\| &|D|^{-\alpha} \sin(t|D|) F_{\lambda} \|_{L^p(B(0,A|t|) \setminus B(0,(A-1)|t|))} \nonumber \\
&\quad \gtrsim_{D_{\lambda}} \left( \left( \int_{B(0,A|t|) \setminus B(0,(A-1)|t|)} |x|^{-\frac{(d-1)p}{2}}\nonumber \right. \right. \\
& \qquad \times \left| \text{Im}(c_+ e^{i\frac{\pi}{2}(\frac{d+1}{2} - \lambda - \alpha)}) (|x| + t)_+^{\lambda + \alpha-\frac{d+1}{2}} \nonumber \right. \\
&\left. \qquad \qquad + \text{Im} (c_+ e^{-i\frac{\pi}{2}(\frac{d+1}{2} - \lambda - \alpha)}) \left(|x| + t\right)_-^{\lambda + \alpha-\frac{d+1}{2}} \right. \nonumber  \\
&\qquad + \text{Im} (c_- e^{i\frac{\pi}{2}(\frac{d+1}{2} - \lambda - \alpha)}) (-|x| + t)_+^{\lambda+\alpha-\frac{d+1}{2}} \nonumber \\
&\left. \left. \qquad \qquad + \text{Im}( c_- e^{-i\frac{\pi}{2}(\frac{d+1}{2} - \lambda - \alpha)}) \left(-|x| + t\right)_-^{\lambda+\alpha-\frac{d+1}{2}} \right|^p dx \right)^{1/p} \nonumber \\
&\hspace{25mm} - \left(\int_{B(0,A|t|) \setminus B(0,(A-1)|t|)} |x|^{-\frac{(d+1)p}{2}} dx\right)^{1/p},
\end{align}
where $\text{Im}(z)$ denotes the imaginary part of $z \in \mathbb{C}$. The problematic part of the integration domain to check whether 
$$|D|^{-\alpha} \sin(t|D|) F_{\lambda}$$
is in $L^p(B(0,A|t|) \setminus B(0,(A-1)|t|))$ corresponds to $\pm|x| \in (-t-\varepsilon,-t+\varepsilon)$, for $0 < \varepsilon \ll 1$. Notice that, by permuting the signs $\pm$, there is no pair in \eqref{eq:cases+-} for which both members get cancelled simultaneously.

The factors $|x|^{-\frac{(d+1)p}{2}}$ and $|x|^{-\frac{(d-1)p}{2}}$ are bounded in this domain. The divergence around $|x| = \pm t$ follows from
\begin{align*}
\left( \lambda + \alpha - \frac{d+1}{2} \right)p + 1 \leq 0 \Longleftrightarrow \lambda + \alpha \leq \frac{d+1}{2} - \frac{1}{p}.
\end{align*}
In summary, the divergence of the norm in $L^p(\mathbb{R}^d)$ follows from $\frac{d-1}{2} < \lambda + \alpha \leq \frac{d+1}{2} - \frac{1}{p}$, with divergence around $\pm |x| = t$.
\end{proof}

\section{Deterministic counterexample on $\mathbb{T}^d$}\label{sec:localization}	
Our next goal is to periodize $F_{\lambda}$ and $|D|^{-\alpha} \sin(t|D|) F_{\lambda}$, and to prove that the periodized counterparts serve also as counterexample for $x \in \mathbb{T}^d$, $1 < p < 2$. Namely, $f_{\lambda}$ is given by
\begin{align}\label{eq:flambdadef-1}
\sum_{k \in \mathbb{Z}^d} F_{\lambda}(x+k), \quad x \in \mathbb{T}^d,
\end{align}
which is an element in $L^p(\mathbb{T}^d)$ for suitable $p$ and $\lambda$, see Lemma \ref{lemma:flambda-wsigmap-per} for details. We will use the periodization strategy used in \cite{W65}, based on the Poisson summation formula. We start giving the main result of this section, which proof is presented below.

\begin{theorem}\label{thm:divergence-torus}
Let $1 < p < 2$, $0 \leq \alpha < (d-1)\left( \frac{1}{p} - \frac{1}{2} \right)$ and $\lambda \in (d - \frac{d}{p} , d)$ such that $\lambda + \alpha \in (\frac{d-1}{2},d)$ (satisfying $\frac{d-1}{2} - \lambda - \alpha \notin -\mathbb{N}$) and $f_{\lambda}$ given in \eqref{eq:flambdadef-1}. Assume also that $\lambda + \alpha \leq \frac{d+1}{2} - \frac{1}{p}$. Then, for any $t \in (-1/2,1/2) \setminus \{0\}$,
\begin{align*}
f_{\lambda} \in L^p(\mathbb{T}^d) \quad \text{and} \quad |D|^{-\alpha} \sin(t|D|) f_{\lambda} \notin L^p(\mathbb{T}^d).
\end{align*}
\end{theorem}

We now provide the analogous result to Proposition \ref{prop:Flambda-inWsigmap} for $f_{\lambda}$. 
\begin{lemma}\label{lemma:flambda-wsigmap-per}
We have that:
\begin{itemize}
\item If $1 \leq p < \infty$ and $\lambda \in (d - \frac{d}{p},d)$, then $f_{\lambda} \in L^p(\mathbb{T}^d)$.
\item If $2 \leq p < \infty$ and $\lambda > d - \frac{d}{p}$, then $f_{\lambda} \in L^p(\mathbb{T}^d)$.
\end{itemize}
\end{lemma}
\begin{proof}
From Proposition \ref{prop:Flambda-inWsigmap}, $F_{\lambda} \in L^1(\mathbb{R}^d)$ as long as $\lambda \in (0,d)$. Therefore, the $L^1(\mathbb{T}^d)$ norm of the Fourier expansion 
\begin{align*}
\sum_{k \in \mathbb{Z}^d} F_{\lambda}(x+k)
\end{align*}
converges and has Fourier coefficients (see Theorem 2.4 from Chapter 7 in \cite{steinweiss})
\begin{align*}
\mathcal{F}[ f_{\lambda}](k) =  h(k) |k|^{-\lambda}, \quad k \in \mathbb{Z}^d.
\end{align*}	
Using Lemma \ref{lemma:Flambda}, dominated convergence theorem and triangle inequality we obtain that, for any $N > d$,
\begin{align*}
\left\| \sum_{k \in \mathbb{Z}^d} F_{\lambda}(x+k) \right\|_{L^p(\mathbb{T}^d)} &\leq \sum_{|k| > 1} \frac{1}{(|k| - 1)^N} + \sum_{k \in \mathbb{Z}^d : |k|  \leq 1} \|F_{\lambda}(x+k) \|_{L^p(\mathbb{T}^d)} \\
&\lesssim \sup_{k \in \mathbb{Z}^d : |k| \leq 1} \|F_{\lambda}(x+k) \|_{L^p(\mathbb{T}^d)}.
\end{align*}
Therefore, the statement $f_{\lambda} \in L^p(\mathbb{T}^d)$ follows from
\begin{align*}
\gamma(d,\lambda) |x + k|^{-d+\lambda} + G_{\lambda,d}(x + k) \in L^p(\mathbb{T}^d), \quad \forall k \in \mathbb{Z}^d \text{ such that } |k| \leq 1,
\end{align*}
and the latter is true if $\lambda > d - \frac{d}{p}$. This proves the first statement.
\hfill \break

Regarding the second statement, let $p>2$. By Hausdorff-Young inequality we have that
\begin{align*}
\|f_{\lambda}\|_{L^p(\mathbb{T}^d)} \lesssim \|(g_{\lambda}(k) e^{ikx})_{k \in \mathbb{Z}^d} \|_{\l^{p'}(\mathbb{Z}^d)} = \left( \sum_{k \in \mathbb{Z}^d} |g_{\lambda}(k)|^{p'} \right)^{\frac{1}{p'}},
\end{align*}
where the last sum converges if $\lambda > d - \frac{d}{p}$. This concludes the proof.
\end{proof}

\begin{proof}[Proof of Theorem \ref{thm:divergence-torus}]
Let $t \neq 0$ and $1 < p < 2$. From Theorem \ref{thm:flambda-tnotzero} we have that 
\begin{align*}
|D|^{-\alpha} \sin(t|D|) F_{\lambda} \in L^1(\mathbb{R}^d)
\end{align*}
if $\lambda + \alpha \in (\frac{d-1}{2},d)$. Thus, applying again Theorem 2.4 from Chapter 7 in \cite{steinweiss} as we similarly did in Lemma \ref{lemma:flambda-wsigmap-per}, we have that the $L^1(\mathbb{T}^d)$ norm of
\begin{align}\label{eq:flambdadef-t}
\sum_{k \in \mathbb{Z}^d} |D|^{-\alpha} \sin(t|D|) F_{\lambda}(x+k)
\end{align}
converges and its Fourier coefficients are given by
\begin{align}\label{eq:periodization-obj2}
\mathcal{F}\left[ \sum_{l \in \mathbb{Z}^d} |D|^{-\alpha} \sin(t|D|)  F_{\lambda}(x+l) \right](k) = |k|^{-\alpha} \sin(t|k|) g_{\lambda}(k), \quad k \in \mathbb{Z}^d.
\end{align} 
Recall that we want to prove the divergence of
\begin{align*}
\||D|^{-\alpha} \sin(t|D|) f_{\lambda} \|^p_{L^p(\mathbb{T}^d)}.
\end{align*}
Therefore, if we show that
\begin{align}\label{eq:finaldiv-2}
\sum_{k \in \mathbb{Z}^d} (|D|^{-\alpha} \sin(t|D|) F_{\lambda})(x+k)
\end{align}
is not in $L^p(\mathbb{T}^d)$, the proof would be concluded. 

\hfill \break

For a.e. $x \in \mathbb{T}^d \setminus \{0\}$, and such that $|x+l| \neq t$ and $-|x+l| \neq t$ for any $l \in \mathbb{Z}^d$, using Theorem \ref{thm:flambda-tnotzero} and the same computation that led us to \eqref{eq:gralint-p-2} we have that
\begin{align}\label{eq:finaldiv-3}
|D&|^{-\alpha} \sin(t|D|) F_{\lambda}(x+k)\nonumber \\
&= -iJ\left(\frac{d-1}{2}-\lambda - \alpha \right) c(d) |x+k|^{-\frac{d-1}{2}} \nonumber \\
&\quad \times  \{ \text{Im}(c_{+} e^{i\frac{\pi}{2}(\frac{d+1}{2} - \lambda - \alpha)}) (|x+k| + t)_+^{\lambda + \alpha-\frac{d+1}{2}} \nonumber \\
&\qquad \qquad + \text{Im}( c_+ e^{-i\frac{\pi}{2}(\frac{d+1}{2} - \lambda - \alpha)}) (|x+k| + t)_-^{\lambda + \alpha-\frac{d+1}{2}} \nonumber \\
&\quad \quad + \text{Im}(c_{-} e^{i\frac{\pi}{2}(\frac{d+1}{2} - \lambda - \alpha)}) (-|x+k| + t)_+^{\lambda + \alpha-\frac{d+1}{2}} \nonumber \\
&\qquad \qquad +  \text{Im}( c_- e^{-i\frac{\pi}{2}(\frac{d+1}{2} - \lambda - \alpha)}) (-|x+k| + t)_-^{\lambda + \alpha-\frac{d+1}{2}}  \}\nonumber  \\
&\quad + \Tilde{c}_{\lambda}(x+k) + \Tilde{D}_{\lambda}(t,x+k),
\end{align}
where $k \in \mathbb{Z}^d$, $\Tilde{D}_{\lambda}(t,\cdot) \in C^{\infty}(\mathbb{R}^d)$, $c_{\pm}$ are the constants from \eqref{eq:cplusminus}, and 
\begin{align*}
&|\Tilde{c}_{\lambda}(x+k)| \lesssim |x+k|^{-\frac{d+1}{2}} \text{ for } \lambda \in \left(\frac{d-1}{2},d\right).
\end{align*}
From the proof of Theorem \ref{thm:nonpropag-linear-rn}, we expect that a singularity around $|x| = \pm |t|$ will give the divergence for $p < 2$ and $\lambda \leq \frac{d+1}{2} - \frac{1}{p}$. Consider $|t| < 1/2$. There exists some $\gamma > 0$ sufficiently small such that $|t| < 1/(2 + 2\gamma)$.  We bound as
\begin{align*}
&\left\| \sum_{k \in \mathbb{Z}^d}  |D|^{-\alpha} \sin(t|D|) F_{\lambda}(x+k)  \right\|_{L^p(\mathbb{T}^d)} \nonumber \\
&\qquad \geq \left\|  |D|^{-\alpha} \sin(t|D|) F_{\lambda}(x)  \right\|_{L^p(B(0,(1 + \gamma)|t|) \setminus B(0,(1 - \gamma)|t|))} \nonumber \\
&\qquad \qquad - \left\| \sum_{k \in \mathbb{Z}^d \setminus \{0\}} |D|^{-\alpha} \sin(t|D|) F_{\lambda}(x+k)  \right\|_{L^p(B(0,(1 + \gamma)|t|) \setminus B(0,(1 - \gamma)|t|))}.
\end{align*}
In order to bound the second norm we use the second point in Theorem \ref{thm:flambda-tnotzero}, so that for any $N > d - \lambda - \alpha$
\begin{align}\label{eq:int1-period}
&\left( \int_{B(0,(1 + \gamma)|t|) \setminus B(0,(1 - \gamma)|t|)} \left| \sum_{k \in \mathbb{Z}^d \setminus \{0\}} |D|^{-\alpha} \sin(t|D|) F_{\lambda}(x+k) \right|^p dx \right)^{1/p} \nonumber \\
&\quad \lesssim \left( \int_{B(0,(1 + \gamma)|t|) \setminus B(0,(1 - \gamma)|t|)} \left( \sum_{k \in \mathbb{Z}^d: |k| = 1} \frac{dx}{|x + k|^{N}} \right)^p dx \right)^{1/p} \nonumber \\
& \qquad \qquad + \left( \int_{B(0,(1 + \gamma)|t|) \setminus B(0,(1 - \gamma)|t|)} \left( \sum_{k \in \mathbb{Z}^d, |k| > 1} \frac{dx}{|x + k|^{N}} \right)^p dx \right)^{1/p}.
\end{align}
Regarding the first integral, since $|k| = 1$, $|t| < 1/(2 + 2\gamma)$ and $|x| \leq (1+\gamma)|t|$,
\begin{align*}
&|x + k|^N \geq (1 - |x|)^N \geq (1 - (1+\gamma)|t|)^N \\
&\quad \Longrightarrow \frac{1}{|x + k|^N} \leq \frac{1}{(1 - (1+\gamma)|t|)^N} < 2^N,
\end{align*}
and thus
\begin{align*}
\left( \int_{B(0,(1 + \gamma)|t|) \setminus B(0,(1 - \gamma)|t|)} \left( \sum_{k \in \mathbb{Z}^d: |k| = 1} \frac{dx}{|x + k|^{N}} \right)^p dx \right)^{1/p} \lesssim 2^N.  
\end{align*}
Regarding the second integral at the right-hand side of \eqref{eq:int1-period},
\begin{align*}
&\left( \int_{B(0,(1 + \gamma)|t|) \setminus B(0,(1 - \gamma)|t|)} \left( \sum_{k \in \mathbb{Z}^d, |k| > 1} \frac{dx}{|x + k|^{N}} \right)^p dx \right)^{1/p} \nonumber \\ 
&\qquad \qquad \qquad \lesssim \sum_{k \in \mathbb{Z}^d: |k| > 1} \frac{1}{(|k| - 1)^N}.
\end{align*}
Thus, if we further assume that $N > d$ then the left-hand side of \eqref{eq:int1-period} can be bounded by
\begin{align*}
2^N + \sum_{k \in \mathbb{Z}^d, |k| > 1} \frac{1}{(|k| - 1)^{N}} \lesssim 1.
\end{align*}
Therefore, the same analysis of the proof for Theorem \ref{thm:nonpropag-linear-rn} for the integral
\begin{align*}
\left( \int_{B(0,(1 + \gamma)|t|) \setminus B(0,(1 - \gamma)|t|)} \left| |D|^{-\alpha} \sin(t|D|) F_{\lambda}(x)  \right|^p dx \right)^{1/p}
\end{align*}
concludes the proof. 
\end{proof}

\section{Proof of Theorem \ref{thm:nonpropag-linear}}
\label{sec:dual}

By using the analysis in previous sections and the uniform boundedness principle, we justify now the extension of the same pathological behaviour of $\langle D \rangle^{-\alpha} \sin(t|D|) : L^p \rightarrow L^p$ to the range $2 < p < \infty$.

We introduce a precise notion of adjoint operator. Let $E,F$ Banach spaces and an unbounded densely defined linear operator $A : E \rightarrow F$ with domain $\mathcal{D}(A)$. We denote by $E^*$ and $F^*$ the corresponding dual spaces, as well as the usual pairing
\begin{align*}
\langle V , U \rangle_{E^*,E} = V(U), \quad \forall V \in E^*, \hspace{2mm} \forall U \in E.
\end{align*}
Denote the adjoint operator as $A^* : F^* \rightarrow E^*$, well-defined on the domain
\begin{align*}
\mathcal{D}(A^*) = \{ v \in F^* : \exists c \geq 0 \text{ such that } |\langle v , A u \rangle_{F^*,F}| \leq c \|u\|_E \hspace{2mm} \forall u \in \mathcal{D}(A) \}.
\end{align*}
In this context, $A^*$ is determined by
\begin{align}\label{eq:arg-dual-0}
\langle v , A u \rangle_{F^*,F} = \langle A^* v , u \rangle_{E^*,E}, \quad \forall v \in \mathcal{D}(A^*), \hspace{2mm} \forall u \in \mathcal{D}(A).
\end{align}
We provide two lemmata that will serve us to prove Theorem \ref{thm:nonpropag-linear}.

\begin{lemma}\label{lemma:dual}
Let $p \in (1 , \infty)$. Let $\mathcal{L} : L^{p'}(\mathbb{T}^d) \rightarrow L^{p'}(\mathbb{T}^d)$ be an unbounded densely defined linear operator. Let $F \in L^{p'}(\mathbb{T}^d)$ such that $\mathcal{L} F \notin L^{p'}(\mathbb{T}^d)$. Then, there exists some $\mathcal{H} \in L^p(\mathbb{T}^d)$ such that $\mathcal{L}^*\mathcal{H} \notin L^p(\mathbb{T}^d)$.
\end{lemma}
\begin{proof}
Let $S : L^{p'}(\mathbb{T}^d) \rightarrow (L^p(\mathbb{T}^d))^*$ be the canonical isometric bijection which associates each $f \in L^{p'}(\mathbb{T}^d)$ with the linear functional
\begin{align}\label{eq:canon-dual}
S(f) : h \in L^p \mapsto S(f)(h) := \int_{\mathbb{T}^d} f h dx.
\end{align}
Moreover, consider the linear operator $S_{\mathcal{L}} : L^{p'}(\mathbb{T}^d) \rightarrow (L^p(\mathbb{T}^d))^*$ given by $S_{\mathcal{L}}(f) = S(\mathcal{L}f)$ for any $f \in L^{p'}(\mathbb{T}^d)$.\\

By the uniform boundedness principle, if for each $h \in L^p(\mathbb{T}^d)$ there exists some finite $C_h > 0$ such that
\begin{align}\label{eq:ubp-hyp}
\sup_{\|f\|_{L^{p'}} \leq 1} |S_{\mathcal{L}}(f)(h)| \leq C_h,
\end{align}
then there exists some finite $C > 0$ such that
\begin{align}\label{eq:ubp-cons}
\sup_{\| f \|_{L^{p'}} \leq 1} \| S_{\mathcal{L}}(f) \|_{(L^p)^*} \leq C.
\end{align}
However, by hypothesis there exists some $F \in L^{p'}(\mathbb{T}^d)$ such that $\mathcal{L} F \notin L^{p'}(\mathbb{T}^d)$. Thus, given that $S$ is an isometry,
\begin{align*}
\| S_{\mathcal{L}}(F / \|F\|_{L^{p'}}) \|_{(L^p)^*} = \| F \|_{L^{p'}}^{-1} \|\mathcal{L}F\|_{L^{p'}} = \infty.
\end{align*}
In particular, \eqref{eq:ubp-cons} does not hold, what implies that the property given in \eqref{eq:ubp-hyp} does not hold neither. Thus, there exists some $\mathcal{H} \in L^p(\mathbb{T}^d)$ such that
\begin{align*}
\sup_{\| f \|_{L^{p'}} \leq 1} |S_{\mathcal{L}}(f)(\mathcal{H})| = \infty.
\end{align*}
By \eqref{eq:arg-dual-0}, this implies that $\mathcal{L}^* \mathcal{H} \notin L^p(\mathbb{T}^d)$.
\end{proof}

\begin{lemma}\label{lemma:LqSq}
Let  $q \in (1,\infty)$ and $t \in \mathbb{R}$. Consider the operator
\begin{align*}
&\mathcal{L}_q := \langle D \rangle^{-\alpha} \sin(t|D|) : L^q(\mathbb{T}^d) \rightarrow L^q(\mathbb{T}^d).
\end{align*}
Then, $\mathcal{L}_q$ is densely defined and closed. Moreover, $\mathcal{L}_q = \mathcal{L}_{q'}^*$.
\end{lemma}

\begin{proof}
Let $q \in (1,\infty)$. Without loss of generality, we identify $L^q(\mathbb{T}^d)$ and $(L^{q'}(\mathbb{T}^d))^*$ via the corresponding canonical isometric bijection between them, like in \eqref{eq:canon-dual}.

Firstly, $\mathcal{L}_q^*$ is well-defined because $\mathcal{L}_q$ is densely defined. Indeed, by considering closures in the $L^q$ topology, we have that $\overline{C^{\infty}(\mathbb{T}^d)} = \overline{\mathcal{D}(\mathcal{L}_q)} = L^q$. Then, $\mathcal{L}_{q}$ is densely defined.

In order to show that $\mathcal{L}_q$ is closed, assume $f \in L^q$ and $(\theta_n)_{n \in \mathbb{N}} \subset \mathcal{D}(\mathcal{L}_q)$ such that $\lim_{n \rightarrow \infty}\theta_n = f$ and $\lim_{n \rightarrow \infty}\mathcal{L}_q\theta_n = \tilde{f}$, both limits in $L^q$. If we prove that $f \in \mathcal{D}(\mathcal{L}_q)$ and $\mathcal{L}_q f = \tilde{f}$, then we would have proved that $\mathcal{L}_q$ is closed. Indeed, considering the usual topology of the distributions space $(C^{\infty}(\mathbb{T}^d))^*$, then $\lim_{n \rightarrow \infty}\theta_n = f$ and $\mathcal{L}_q : (C^{\infty}(\mathbb{T}^d))^* \rightarrow (C^{\infty}(\mathbb{T}^d))^*$ is continuous. By uniqueness of limits in $(C^{\infty}(\mathbb{T}^d))^*$ we get $\mathcal{L}_q f = \tilde{f} \in L^q$, as well as $f \in \mathcal{D}(\mathcal{L}_q)$.

We prove now that $\mathcal{L}_{q'}^* = \mathcal{L}_q$. Let $h \in \mathcal{D}(\mathcal{L}_q)$ and $f \in \mathcal{D}(\mathcal{L}_{q'})$. Consider a standard mollifier $\psi \in C^{\infty}(\mathbb{T}^d)$ such that $0 \leq \psi \leq 1$, $\int_{\mathbb{T}^d} \psi(x) dx = 1$, and denote $\psi_n(x) = n^d \psi(nx)$ for any $n \in \mathbb{N}$. Consider the sequences of functions $h_n = \psi_n \ast h$ and $f_n = \psi_n \ast f$, which satisfy
\begin{align}\label{eq:limits-hnfn}
\lim_{n \rightarrow \infty}f_n = f \text{ in }L^{q'}(\mathbb{T}^d), \quad \lim_{n \rightarrow \infty} h_n = h \text{ in } L^q(\mathbb{T}^d).
\end{align}
Notice that $\mathcal{L}_q$ is a translation-invariant operator, so that for all $n \in \mathbb{N}$
\begin{align*}
\mathcal{L}_q h_n = \psi_n \ast (\mathcal{L}_q h) \quad \text{ and } \quad \mathcal{L}_q f_n = \psi_n \ast (\mathcal{L}_q f).
\end{align*}
Thus, given that $\mathcal{L}_q h \in L^q(\mathbb{T}^d)$ and $\mathcal{L}_{q'} f \in L^{q'}(\mathbb{T}^d)$, if we do the limit $n \rightarrow \infty$ under the topology of $L^q(\mathbb{T}^d)$ and $L^{q'}(\mathbb{T}^d)$ respectively we obtain that
\begin{align}\label{eq:limits-hnfn-L}
\lim_{n \rightarrow \infty} \mathcal{L}_q h_n = \mathcal{L}_q h \quad \text{ in } L^q(\mathbb{T}^d) \quad \text{ and } \quad \lim_{n \rightarrow \infty} \mathcal{L}_{q'} f_n = \mathcal{L}_{q'} f \quad \text{ in } L^{q'}(\mathbb{T}^d).
\end{align}
Thus, we have that
\begin{align}\label{eq:density-arg}
&\langle h , \mathcal{L}_{q'} f \rangle_{(L^{q'})^*,L^{q'}} = \int_{\mathbb{T}^d} h \cdot \mathcal{L}_{q'} f dx = \lim_{n \rightarrow \infty} \int_{\mathbb{T}^d} h_n \cdot \mathcal{L}_{q'} f_n dx, \nonumber \\
&\langle f , \mathcal{L}_q h \rangle_{(L^q)^*,L^q} = \int_{\mathbb{T}^d} f \cdot \mathcal{L}_q h dx = \lim_{n \rightarrow \infty} \int_{\mathbb{T}^d} f_n \cdot \mathcal{L}_q h_n dx.
\end{align}
The first limit follows by triangle and Hölder's inequalities since
\begin{align*}
&\left| \int_{\mathbb{T}^d} h \cdot \mathcal{L}_{q'} f dx - \int_{\mathbb{T}^d} h_n \cdot \mathcal{L}_{q'} f_n dx \right| \\
& \qquad \leq \left| \int_{\mathbb{T}^d} (h - h_n) \cdot \mathcal{L}_{q'} f dx \right| + \left| \int_{\mathbb{T}^d} h_n \cdot \mathcal{L}_{q'} (f_n -  f) dx \right| \\
& \qquad \leq \| h - h_n \|_{L^q} \|\mathcal{L}_{q'} f\|_{L^{q'}} + \| h_n \|_{L^q} \| \mathcal{L}_{q'} (f_n - f) \|_{L^{q'}},
\end{align*}
where the last quantity converges to $0$ as $n \rightarrow \infty$ due to the convergence in \eqref{eq:limits-hnfn} and \eqref{eq:limits-hnfn-L}. The second limit in \eqref{eq:density-arg} follows in a similar way. 
\hfill \break

Moreover, by Plancherel theorem we have that
\begin{align}\label{eq:arg-dual-2}
&\int_{\mathbb{T}^d} h_n \cdot \mathcal{L}_{q'} f_n dx = \sum_{k \in \mathbb{Z}^d} \widehat{h_n}(k) (\langle k \rangle^{- \alpha} \sin(t|k|) \widehat{f_n}(k)) \nonumber \\
&\quad  = \sum_{k \in \mathbb{Z}^d} (\langle k \rangle^{- \alpha} \sin(t|k|) \widehat{h_n}(k) ) \widehat{f_n}(k) = \int_{\mathbb{T}^d} \mathcal{L}_q h_n \cdot f_n dx, \quad \forall  n \in \mathbb{N}.
\end{align}
Therefore, plugging \eqref{eq:arg-dual-2} into \eqref{eq:arg-dual-0}, and using both limits in \eqref{eq:density-arg}, we get
\begin{align*}
\int_{\mathbb{T}^d} \mathcal{L}^*_{q'} h \cdot f dx = \int_{\mathbb{T}^d} h \cdot \mathcal{L}_{q'} f dx = \int_{\mathbb{T}^d} \mathcal{L}_q h \cdot f dx.
\end{align*}
In particular, given $h \in \mathcal{D}(\mathcal{L}_q)$, for any $\mu \in C^{\infty}(\mathbb{T}^d)$ we have proved that
\begin{align*}
\int_{\mathbb{T}^d} (\mathcal{L}^*_{q'} - \mathcal{L}_q) h \cdot \mu dx = 0,
\end{align*}
so $\mathcal{L}^*_{q'} h =  \mathcal{L}_q h$ almost everywhere on $\mathbb{T}^d$. By a similar argument, the same equality follows by choosing $h \in \mathcal{D}(\mathcal{L}_{q'}^*)$. Thus, $\mathcal{D}(\mathcal{L}_q) =  \mathcal{D}(\mathcal{L}_{q'}^*)$. This concludes the proof of $\mathcal{L}_q = \mathcal{L}_{q'}^*$.
\end{proof}

\begin{proof}[Proof of Theorem \ref{thm:nonpropag-linear}]
Let $t \in (-1/2,1/2) \setminus \{0\}$ and $\mathcal{L}_p$ defined as in Lemma \ref{lemma:LqSq}. We first prove the statement with the operator $\mathcal{L}_p$, and postpone to the end of the proof the justification of why this can be transferred to the operator
\begin{align}\label{eq:genuine-op}
\langle D \rangle^{-\alpha} \frac{\sin(t|D|)}{|D|} \langle D \rangle : L^p(\mathbb{T}^d) \rightarrow L^p(\mathbb{T}^d).
\end{align}

\noindent
The case $1 < p < 2$ for $\mathcal{L}_p$ has been proved in Theorem \ref{thm:divergence-torus}. We consider
\begin{itemize}
\item $2 < p < \infty$ (so $1 < p' < 2$), and $\alpha \geq 0$ such that
\begin{align*}
\alpha < (d-1)\left(\frac{1}{p'} - \frac{1}{2} \right) = (d-1) \left( \frac{1}{2} - \frac{1}{p} \right); 
\end{align*}
\item and $\lambda \in (d - \frac{d}{p'} , d)$ such that $\lambda + \alpha
 \leq \frac{d+1}{2} - \frac{1}{p'}$, $\lambda + \alpha \in ( \frac{d-1}{2},d)$ and $\frac{d-1}{2} - \lambda - \alpha \notin -\mathbb{N}$.
\end{itemize}
Therefore, by Theorem \ref{thm:divergence-torus} we know that the functions $f_{\lambda}$ defined in \eqref{eq:flambdadef-1} satisfy
\begin{align}\label{eq:dual-pre}
f_{\lambda} \in L^{p'}(\mathbb{T}^d) \quad \text{ and } \quad \mathcal{L}_{p'} f_{\lambda} \notin L^{p'}(\mathbb{T}^d).
\end{align} 

By \eqref{eq:dual-pre}, we can apply Lemmata \ref{lemma:dual} and \ref{lemma:LqSq} with $F = f_{\lambda}$ and $q = p$ in order to get the result with $\mathcal{L}_p$, i.e. that there exists some $\mathcal{H} \in L^p(\mathbb{T}^d)$ such that $\langle D \rangle^{-\alpha} \sin(t|D|) \mathcal{H} \notin L^p(\mathbb{T}^d)$. 

\hfill \break
\noindent
To conclude, we justify the transference of the result to the operator \eqref{eq:genuine-op} for any $p \in (1 , 2) \cup (2 , \infty)$. Notice that
\begin{align}\label{eq:bounded-DD}
|D|/\langle D \rangle : L^q(\mathbb{T}^d) \rightarrow L^q(\mathbb{T}^d)
\end{align} 
is a bounded operator for any $1 \leq q \leq \infty$ (see Theorem 3.6 of Chapter VII in \cite{steinweiss} and Lemma 2 in page 133 of \cite{S70}). We proceed by contradiction. Assume that
\begin{align}\label{eq:cont-H-1}
\langle D \rangle^{-\alpha} \frac{\sin(t|D|)}{|D|} \langle D \rangle \mathcal{H} \in L^p(\mathbb{T}^d).
\end{align} 
Then, since \eqref{eq:bounded-DD} is bounded for any $1 \leq q \leq \infty$, we have that
\begin{align}\label{eq:cont-H}
\mathcal{H} \in \mathcal{D}\left( \left( \frac{|D|}{\langle D \rangle} \right)\langle D \rangle^{-\alpha} \frac{\sin(t|D|)}{|D|} \langle D \rangle  \right).
\end{align}
Let $\tilde{h}_n = \psi_n \ast \mathcal{H}$ with $\psi_n$ defined in the proof of Lemma \ref{lemma:LqSq}, so that $\lim_{n \rightarrow \infty} \tilde{h}_n = \mathcal{H}$ in $L^p(\mathbb{T}^d)$. In particular, by \eqref{eq:cont-H-1} and the fact that the operator \eqref{eq:genuine-op} is translation-invariant,
\begin{align*}
\lim_{n \rightarrow \infty} \langle D \rangle^{-\alpha} \frac{\sin(t|D|)}{|D|} \langle D \rangle \tilde{h}_n = \langle D \rangle^{-\alpha} \frac{\sin(t|D|)}{|D|} \langle D \rangle \mathcal{H} \quad \text{ in } L^p(\mathbb{T}^d).
\end{align*}
This convergence together with Plancherel theorem applied as before and the boundedness of \eqref{eq:bounded-DD} implies\footnote{Notice that \eqref{eq:cont-H-2} holds whether $\widehat{\mathcal{H}}(0) = 0$ or $\widehat{\mathcal{H}}(0) \neq 0$, since the kernel of both operators involved in \eqref{eq:cont-H-2} contain the set of constant functions.} the following limit under the topology of $L^p(\mathbb{T}^d)$:
\begin{align}\label{eq:cont-H-2}
&\lim_{n \rightarrow \infty} \mathcal{L}_p \tilde{h}_n = \lim_{n \rightarrow \infty}  \left( \frac{|D|}{\langle D \rangle} \right) \langle D \rangle^{-\alpha} \frac{\sin(t|D|)}{|D|} \langle D \rangle \tilde{h}_n \nonumber \\
&\qquad \qquad \qquad = \left( \frac{|D|}{\langle D \rangle} \right) \langle D \rangle^{-\alpha} \frac{\sin(t|D|)}{|D|} \langle D \rangle \mathcal{H} \in L^p(\mathbb{T}^d).
\end{align}
However, since $\mathcal{L}_p$ is closed this would imply that $\mathcal{H} \in \mathcal{D}(\mathcal{L}_p)$, but we have proved that this is not true. Therefore
\begin{align*}
\langle D \rangle^{-\alpha} S(t) (0,\langle D \rangle \mathcal{H}) \notin L^p(\mathbb{T}^d).
\end{align*} 
\end{proof}

\section{Proof of Theorem \ref{thm:nonpropag-nonlinear}}\label{sec:proof-nonlinearcex}
Fix $d = 3$. All we need is to prove an appropriate control of \eqref{eq:quantity-to-control}.
Given $r \in \mathbb{R}$, $p,q \in [1,\infty]$, $T>0$, we denote 
\begin{align*}
L^q_T W^{r,p}(\mathbb{T}^3) = L^q([-T,T] , W^{r,p}(\mathbb{T}^3)),
\end{align*}
being the external Lebesgue space given for the time variable, and the internal space for the space variable. We will use the following Strichartz estimates on $\mathbb{T}^3$, which can be directly obtained in the periodic setting from the well-known estimates on $\mathbb{R}^3$ from \cite{GV95} via finite speed of propagation (see for instance Proposition 4.1.10 in \cite{GST19}). Notice that this transference to the torus via finite speed of propagation requires to impose $|t| < 1$ in all the upcoming analysis.

\begin{proposition}\label{prop:st}
		Let $\mu \in \mathbb{R}$, $T \in (0,1)$ and $2 \leq p,q \leq \infty$ be such that
		\begin{align}\label{eq:adm-st}
			\frac{2}{q} \leq 1 - \frac{2}{p}, \quad (q,p) \neq (2,\infty), \quad \frac{3}{2} - \mu \leq \frac{3}{p} + \frac{1}{q}.
		\end{align} 
		We say that $(q,p)$ are $\mu$-admissible if they satisfy these relations. Then
		\begin{align*}
			\left\|\frac{\sin(t|D|)}{|D|} f\right\|_{L^q_TL^p(\mathbb{T}^3)} \lesssim \|f\|_{H^{\mu-1}(\mathbb{T}^3)}. 
		\end{align*}
	\end{proposition}
	
	The control of the Duhamel integral in \eqref{eq:quantity-to-control} is given in the following lemma.
	
	\begin{lemma}\label{lemma:infcontrol-duhamel-st}
	Let $T \in (0,1)$, $p \geq 2$, $\sigma > \frac{3}{p} - \frac{1}{2}$. Then
		\begin{align}\label{eq:control-duh}
			\left\| \int_0^t \frac{\sin((t-\tau)|D|)}{|D|} u^3(\tau) d\tau \right\|_{L^{\infty}_T W^{\sigma,p}(\mathbb{T}^3)} \lesssim T \sup_{|t| \leq T} \| u(t) \|_{L^{\infty}(\mathbb{T}^3)}^2 \|u(t)\|_{H^{\sigma + \frac{1}{2} - \frac{3}{p}}(\mathbb{T}^3)}.
		\end{align}
		Moreover
		\begin{align}\label{eq:control-duh-1}
		\left\| \int_0^t \frac{\sin((t-\tau)|D|)}{|D|} F(\tau) d\tau \right\|_{L^{\infty}_T W^{\sigma,p}(\mathbb{T}^3)} \lesssim  \|F\|_{L^1_T H^{\sigma + \frac{1}{2} - \frac{3}{p}}(\mathbb{T}^3)}.
		\end{align}
	\end{lemma}
	\begin{proof}
	Let $(q,p)$ be $\mu$-admissible in the sense of \eqref{eq:adm-st}. Then, by the triangle and Minkowski inequalities,
	\begin{align*}
			&\left\| \int_0^t \frac{\sin((t-\tau)|D|)}{|D|} F(\tau) d\tau \right\|_{L^q_T W^{\sigma,p}(\mathbb{T}^3)}  \nonumber\\
			&\leq \left\| \left\| \left\|\frac{\sin((t-\tau)|D|)}{|D|} F(\tau)\right\|_{W^{\sigma,p}(\mathbb{T}^3)} \right\|_{L^1_{\tau}[-T,T]} \right\|_{L^q_t[-T,T]}  \nonumber\\
			&\leq \left\|  \left\|\frac{\sin((t-\tau)|D|)}{|D|} F(\tau)\right\|_{L^q_T W^{\sigma,p}(\mathbb{T}^3)}  \right\|_{L^1_{\tau}[-T,T]} \nonumber\\
			&\lesssim \left\| \left\| F(\tau) \right\|_{H^{\sigma + \mu -1}(\mathbb{T}^3)} \right\|_{L^1_{\tau}[-T,T]} = \|F\|_{L^1_T H^{\sigma + \mu - 1}(\mathbb{T}^3)}.
		\end{align*}
		Considering the sharp case in the last inequality of \eqref{eq:adm-st},
		\begin{align*}
		\mu = \frac{3}{2} - \frac{3}{p} - \frac{1}{q} \Longrightarrow \mu+\sigma-1 = \sigma + \frac{1}{2} - \frac{3}{p} - \frac{1}{q}.
		\end{align*}
		Take $q = \infty$ to obtain \eqref{eq:control-duh-1}. Furthermore, if we consider $\sigma + \frac{1}{2} - \frac{3}{p} > 0$ and $F = u^3$ we can apply Leibniz fractional rule \cite{grafakos_kp} to obtain
		\begin{align*}
		\|F\|_{L^1_T H^{\sigma + \frac{1}{2} - \frac{3}{p}}(\mathbb{T}^3)} \leq T \sup_{|t| \leq T} \|u(t)\|^2_{L^{\infty}(\mathbb{T}^3)} \|u(t)\|_{H^{\sigma + \frac{1}{2} - \frac{3}{p}}(\mathbb{T}^3)}.
		\end{align*}
	\end{proof}

	\begin{proof}[Proof of Theorem \ref{thm:nonpropag-nonlinear}]
	Recall that $p > 2$, $\alpha \in [\max\{0,\frac{1}{2} - \frac{3}{p}\},1 - \frac{2}{p})$ and $\sigma \in \mathbb{R}$ is such that $\sigma + \alpha > \frac{3}{2}$. Let $t \in (-1/2,1/2) \setminus \{0\}$. Moreover, let $\mathcal{H}$ be the function from Theorem \ref{thm:nonpropag-linear}. As we previously said, the density of $\Theta_{\mathcal{H}}$ in $W^{\sigma + \alpha,p}(\mathbb{T}^3)$ follows from the density of $C^{\infty}(\mathbb{T}^3)$.
	
	We have that $\langle D \rangle^{-\sigma} \mathcal{H} \in H^{\sigma + \alpha}(\mathbb{T}^3)$. With our choice of $\alpha$ and $\sigma$ we have that $\sigma + \alpha \geq 1$, which suffices to guarantee the existence of a unique global solution $u$ of \eqref{eq:cauchywave} in the space $C(\mathbb{R},H^{\sigma + \alpha}(\mathbb{T}^3))$ (see \cite{GST19}).
	
	Furthermore, since by Theorem \ref{thm:nonpropag-linear} we have
	\begin{align*}
	S(t)( 0 ,  \langle D \rangle ( \langle D \rangle^{-\sigma} \mathcal{H})) \notin W^{\sigma,p}(\mathbb{T}^3),
	\end{align*}
	then it suffices to prove
	\begin{align}\label{eq:mainproof-2}
	\left\| \int_0^t \frac{\sin((t-\tau)|D|)}{|D|} \right. & \left. u^3(\tau) d\tau \right\|_{L^{\infty}_T W^{\sigma,p}(\mathbb{T}^3)} < \infty
	\end{align}
	for $T \in (0,1)$. In other words, this would show that there exists a unique solution $u$ in $C(\mathbb{R} , H^{\sigma + \alpha}(\mathbb{T}^3))$ such that $(u,\partial_t u)|_{t=0} = (0,\langle D \rangle g)$ and
	\begin{align*}
	u(t,\cdot) \notin W^{\sigma,p}(\mathbb{T}^3).
	\end{align*}
	
	Let $T \in (0,1)$. By Lemma \ref{lemma:infcontrol-duhamel-st} (notice that $\sigma > \frac{3}{2} - \alpha > \frac{1}{2} + \frac{2}{p} > \frac{3}{p} - \frac{1}{2}$) we have that
	\begin{align}\label{eq:mainproof-1}
	\left\| \int_0^t \frac{\sin((t-\tau)|D|)}{|D|} \right. & \left. u^3(\tau) d\tau \right\|_{L^{\infty}_T W^{\sigma,p}(\mathbb{T}^3)} \nonumber \\
	&\lesssim T  \| u \|_{L^{\infty}_TL^{\infty}(\mathbb{T}^3)}^2 \|u\|_{L^{\infty}_T H^{\sigma  + \frac{1}{2} - \frac{3}{p}}(\mathbb{T}^3)}.
	\end{align}
	From $\alpha \geq \frac{1}{2} - \frac{3}{p}$ we have that
	\begin{align*}
	\|u\|_{L^{\infty}_T H^{\sigma  + \frac{1}{2} - \frac{3}{p}}(\mathbb{T}^3)} \leq \|u\|_{L^{\infty}_T H^{\sigma  + \alpha}(\mathbb{T}^3)} < \infty.
	\end{align*}			
	Furthermore, given $\sigma + \alpha > \frac{3}{2}$ and the Sobolev embedding $H^{\frac{3}{2} + \varepsilon}(\mathbb{T}^3) \subset L^{\infty}(\mathbb{T}^3)$ for any $\varepsilon > 0$, we have that
	\begin{align*}
	\| u \|_{L^{\infty}_TL^{\infty}(\mathbb{T}^3)} \lesssim \| u \|_{L^{\infty}_T H^{\sigma + \alpha}(\mathbb{T}^3)} < \infty.
	\end{align*}
	\end{proof}

	\appendix
	
		\section{Appendix: Bessel functions}\label{app:bessel}
We define Bessel functions in the following way. See \cite{watson} or Chapter $4$ of \cite{steinweiss} for further detail. 
	\begin{definition}\label{lemma:bessel_kreal}
		Let $k > -1/2$ and $t > 0$. Then
		\begin{align}\label{eq:bessel2}
			J_k(t) = \frac{(t/2)^k}{\Gamma(2k+1)\Gamma(1/2)} \int_{-1}^1 e^{its} (1-s^2)^{(2k-1)/2}ds.
		\end{align}
	\end{definition}
	
Moreover, for $t$ large we have the following asymptotic behaviour:
		\begin{align}\label{eq:asympBessel}
			J_k(t) &= \left(\frac{2}{\pi}\right)^{1/2} t^{-1/2} \cos \left( t - \frac{\pi k}{2} - \frac{\pi}{4} \right) + R_k(t) \nonumber \\
			&=t^{-1/2} ( c_+ e^{it} + c_- e^{-it} ) + R_k(t),
		\end{align}
where $R_k(t) = \mathcal{O}(t^{-\frac{3}{2}})$ and
\begin{align}\label{eq:cplusminus}
c_+ = \overline{c_-} = \frac{e^{-i\pi\left(\frac{k}{2} + \frac{1}{4}\right)}}{(2\pi)^{1/2}}.
\end{align}

We present the proof of Lemma \ref{lemma:ftransformradial_bessel}, which justifies the application of Bessel functions in this work. 
	\begin{proof}[Proof of Lemma \ref{lemma:ftransformradial_bessel}]
	The case $f \in L^1(\mathbb{R}^d)$ is proved in page 155 of \cite{steinweiss}. To extend this result to any $f$ in $L^p(\mathbb{R}^d)$, $1 < p < \infty$, it suffices to use a density argument with radial mollifiers. Indeed, let $\phi \in C^{\infty}_c(\mathbb{R}^d)$ be radial and denote $\phi_{\varepsilon} = \varepsilon^{-d}\phi(\cdot/\varepsilon)$ for any $\varepsilon > 0$, so that $f_{\varepsilon} = \phi_{\varepsilon} \ast f \rightarrow f$ as $\varepsilon \rightarrow 0^+$ in $L^p(\mathbb{R}^d)$. In particular, $f_{\varepsilon}(x) \rightarrow f(x)$ for a.e. $x \in \mathbb{R}^d$ up to a subsequence. Observe that $f_{\varepsilon}$ is radial because $f$ and $\phi_{\varepsilon}$ are radial, so we can write $f_{\varepsilon}(x) = f_{\varepsilon,0}(|x|)$ for some $f_{\varepsilon,0}$, and apply the first part of the proof to each $\mathcal{F}[f_{\varepsilon}]$. Thus, $f_{\varepsilon,0}(|x|) \rightarrow f_0(|x|)$ for a.e. $x \in \mathbb{R}^d$ and $\varepsilon \rightarrow 0^+$, up to a subsequence.
	
Thanks to the dominated convergence theorem, for any $\varphi \in \mathcal{S}(\mathbb{R}^d)$ we have that
	\begin{align*}
	\int_{\mathbb{R}^d} f_{\varepsilon}(x) \varphi(x) dx \rightarrow \int_{\mathbb{R}^d} f(x) \varphi(x) dx \text{ as } \varepsilon \rightarrow 0^+.
	\end{align*}
	Moreover, 
	\begin{align*}
	&\int_{\mathbb{R}^d} f_{\varepsilon}(x) \varphi(x) dx = \int_{\mathbb{R}^d} \mathcal{F}[f_{\varepsilon}](\xi) \mathcal{F}[\varphi](\xi) d\xi \nonumber \\
	& \int_{\mathbb{R}^d} \left( c(d) |\xi|^{-\frac{d-2}{2}} \int_0^{\infty} f_{\varepsilon,0}(\rho)  J_{\frac{d-2}{2}}(\rho|\xi|) \rho^{d/2} d\rho \right) \mathcal{F}[\varphi](\xi) d\xi.
	\end{align*}
	where we used Plancherel theorem and the known result in $L^1(\mathbb{R}^d)$. Since the integral
	\begin{align*}
	\int_0^{\infty} f_0(\rho) J_{\frac{d-2}{2}}(\rho|\xi|) \rho^{d/2} d\rho
	\end{align*}
	exists for a.e. $\xi \in \mathbb{R}^d$, by the a.e. pointwise convergence given above $f_{\varepsilon,0}(\rho) \rightarrow f_0(\rho)$ and by dominated convergence theorem we have that
	\begin{align*}
	\mathcal{F}[f](\xi) = c(d) |\xi|^{-\frac{d-2}{2}} \int_0^{\infty} f_0(\rho) J_{\frac{d-2}{2}}(\rho|\xi|) \rho^{d/2} d\rho
	\end{align*}
	for a.e. $\xi \in \mathbb{R}^d$.
	\end{proof}

	\section{Appendix: Homogeneous distributions}\label{app:temp-dist-xalpha+}
Let $\alpha \in \mathbb{C}$ and $\varphi \in \mathcal{S}(\mathbb{R})$. Observe that
\begin{align*}
		\int_0^{\infty} \mathbb{I}_{(1,\infty)}(x) x^{\alpha} \varphi(x) dx
	\end{align*}
	is well-defined for any complex value of $\alpha$. In this section we will compute the inverse distributional Fourier transform of $\mathbb{I}_{(1,\infty)}(x) x^{\alpha}$, i.e. we will prove Lemma \ref{lemma:invft-xalpha-nu}. We need a preliminary result.

\begin{lemma}\label{lemma:limit-complex}
	Let $\theta, \beta \in \mathbb{R}$. Then
	\begin{align*}
	\lim_{\gamma \rightarrow 0^+} (\theta + i \gamma)^{\beta} = \left\{ \begin{array}{lcc} \theta^{\beta} & \text{for} & \theta > 0 \\ \\ (-1)^{\beta}\theta^{\beta} e^{i\beta\pi} & \text{for} & \theta < 0. \end{array} \right.
	\end{align*}
	\end{lemma}
	
\begin{proof}
	Observe that
	\begin{align*}
		\lim_{\tau \rightarrow 0^+} &(\theta + i \tau)^{\beta} = \exp[\beta(\log|\theta + i\tau| + i \text{arg}(\theta+i\tau))] \nonumber\\
		&= \left\{ \begin{array}{lcc} \theta^{\beta} & \text{for} & \theta > 0 \\ \\ (-1)^{\beta}\theta^{\beta} e^{i\beta \pi} & \text{for} & \theta < 0, \end{array} \right. 
	\end{align*}
	given that $\text{arg}(\theta + i\tau) \overset{\tau \rightarrow 0^+}{\rightarrow} 0$ for $\theta > 0$ and $\text{arg}(\theta + i\tau) \overset{\tau \rightarrow 0^+}{\rightarrow} \pi$ for $\theta < 0$.
	\end{proof}
	
	\begin{proof}[Proof of Lemma \ref{lemma:invft-xalpha-nu}]
	Consider
	\begin{align*}
		G_{\tau}(x) = x^{\alpha} e^{-\tau x} \mathbb{I}_{(1,\infty)}(x),
	\end{align*}
	for $\tau > 0$, so that we will compute $\mathcal{F}^{-1}[G_{\tau}]$ and take the limit $\tau \rightarrow 0^+$ in $\mathcal{S}'(\mathbb{R})$. Given $\theta \in \mathbb{R}$, write
	\begin{align*}
		\mathcal{F}^{-1}[G_{\tau}](\theta) = \int_{1}^{\infty} x^{\alpha} e^{-\tau x} e^{ix\theta} dx = \int_{1}^{\infty} x^{\alpha} e^{i z x} dx, 
	\end{align*}
	for $z = \theta + i \tau$. This implies that $z$ is in the upper half-plane of $\mathbb{C}$, so that $\text{arg}(z) \in (0,\pi)$. We change variables through the rotation $\xi = - i z x$, so $\text{arg}(\xi) = \text{arg}(z)  - \frac{\pi}{2} \in (-\pi/2,\pi/2)$ and
	\begin{align*}						
		\int_{1}^{\infty} x^{\alpha} e^{i z x} dx = \left( \frac{e^{i\pi/2}}{z} \right)^{\alpha+1} \int_{e^{i(\text{arg}(z)-\pi/2)}(1,\infty)} \xi^{\alpha} e^{-\xi} d\xi.
	\end{align*}
	The next step is to prove that we can remove the rotation $e^{i(\text{arg}(z)-\pi/2)}$ from the integration domain. Consider within the complex plane a contour $L$ surrounding the sector between angles $0$ and $\text{arg}(z)-\pi/2$. This contour departs from radius $1$ and reaches some radius $M>0$ which we will make arbitrarily large eventually. 

Assume without loss of generality that $\text{arg}(z) \in (\frac{\pi}{2},\pi)$ (the case $\text{arg}(z) \in (0,\frac{\pi}{2})$ is proved in a similar way). We write $L = L_1 + L_2 + L_3 + L_4$, where
		\begin{align*}
		&L_1 = \{ x : x \in (1,M) \}, \quad L_2 = \{ e^{i(\text{arg}(z)-\pi/2)}(M + 1 - x) : x \in (1,M) \}, \nonumber \\
			&L_3 = \{ M e^{i\gamma} : \gamma \in (0,\text{arg}(z)-\pi/2)\}, \nonumber \\
			& L_4 = \{e^{i(\text{arg}(z) - \pi/2-\gamma)} : \gamma \in (0,\text{arg}(z)-\pi/2)\}.
	\end{align*}
Given that $\xi^{\alpha} e^{-\xi}$ is holomorphic in $\{\xi \in \mathbb C : |\xi| > 0\}$, by Cauchy's theorem we have that
	\begin{align*}
		\int_L \xi^{\alpha} e^{-\xi} d\xi = 0.
	\end{align*}
	The term from that integral corresponding to a constant radius $M$ satisfies
	\begin{align*}
		&\left|\int_{L_3}\xi^{\alpha} e^{-\xi} d\xi \right| = \left| \int_0^{\text{arg}(z)-\pi/2} M^{\alpha+1} e^{i (\alpha + 1) \gamma} e^{-M(\cos \gamma + i \sin \gamma)} i d\gamma \right|  \nonumber\\
		&\qquad \leq M^{\text{Re}(\alpha) + 1} \left( \text{arg}(z)-\frac{\pi}{2} \right)\sup_{\gamma \in (0,\text{arg}(z)-\pi/2)} e^{- M \cos(\gamma)},
	\end{align*}
	which tends to $0$ as $M \rightarrow \infty$. On the other hand, the sector of constant radius $1$ satisfies
	\begin{align*}
		&\left|\int_{L_4}\xi^{\alpha} e^{-\xi} d\xi \right| \nonumber \\
		&\qquad = \left| \int_0^{\text{arg}(z) - \pi/2}  e^{-i (\alpha + 1) \gamma} e^{-(\cos(\text{arg}(z) - \pi/2-\gamma) + i \sin(\text{arg}(z) - \pi/2-\gamma))} i d\gamma \right|  \nonumber\\
		&\qquad \leq  \left( \text{arg}(z) - \frac{\pi}{2} \right) \sup_{\gamma \in (0,\text{arg}(z)-\frac{\pi}{2})} e^{- \cos(\text{arg}(z) - \pi/2-\gamma)} \leq \frac{\pi}{2}.
	\end{align*}
	Therefore,
	\begin{align*}
	\mathcal{F}^{-1}[G_{\tau}](\theta) = \left(\int_0^{\infty} \xi^{\alpha} e^{-\xi} d\xi - \int_0^1 \xi^{\alpha} e^{-\xi} d\xi - \int_{L_4}\xi^{\alpha} e^{-\xi} d\xi \right)\left( \frac{e^{i\pi/2}}{\theta + i \tau} \right)^{\alpha+1}.
	\end{align*}
	Observe that
\begin{align*}
\max\left\{ \left|\int_0^{\infty} \right. \right. & \left. \left. e^{-\xi} \xi^{\alpha} d\xi \right| , \left|\int_0^{1} e^{-\xi} \xi^{\alpha} d\xi \right| \right\} \\
&\leq \int_0^{\infty} e^{-\xi} \xi^{\text{Re}(\alpha)} d\xi \leq \Gamma(\text{Re}(\alpha) + 1),
\end{align*}
where the latter is well-defined as long as we consider $\text{Re}(\alpha) \notin -\mathbb{N}$ and an analytic continuation of $\Gamma$ (see e.g. \cite{steinshak-complex}).
All in all,
\begin{align}\label{eq:ftransform-x+exp-translation}
\mathcal{F}^{-1}[G_{\tau}](\theta) = J(\alpha)  \left( \frac{e^{i\pi/2}}{\theta + i \tau} \right)^{\alpha+1},
\end{align}
where $J(\alpha)$ is a constant only depending on $\alpha$ such that
\begin{align*}
|J(\alpha)| \leq  \Gamma(\text{Re}(\alpha)+1) + \frac{\pi}{2}.
\end{align*}
If we take the limit $\tau \rightarrow 0^+$ in $\mathcal{S}'(\mathbb{R})$, from Lemma \ref{lemma:limit-complex} we have that
\begin{align}\label{eq:ftransform_x+-translation}
&\mathcal{F}^{-1}[x^{\alpha}\mathbb{I}_{(1,\infty)}(x)](\theta) \nonumber \\
&\qquad = J(\alpha)  [e^{i\pi(\alpha+1)/2} (\theta \mathbb{I}_{(0,\infty)}(\theta))^{-\alpha-1} + e^{-i\pi(\alpha+1)/2} (|\theta| \mathbb{I}_{(-\infty,0)}(\theta))^{-\alpha-1}].
\end{align}
\end{proof}

\section*{Acknowledgments}
	The author is supported by the Basque Government through the program BERC 2026-2029 (BCAM), by the project PID2024-156169NB-I00 (NCAFA), by the Severo Ochoa accreditation CEX2021-001142-S (BCAM), and by the predoctoral program of the Education Department of the Basque Government.
	
	The author thanks Nikolay Tzvetkov for proposing this interesting problem during his stay at L'École Normale Supérieure de Lyon, as well as his PhD supervisor Renato Lucà for useful advices, and Chenmin Sun, who noticed us about some serious corrections needed for Theorems \ref{thm:nonpropag-linear-rn} and \ref{thm:divergence-torus}.



\begin{thebibliography}{10}
	
		\bibitem{AT08}
		A.~Ayache and N.~Tzvetkov.
		\newblock The $L^p$ properties of Gaussian random series.
		\newblock {\em Trans. Am. Math.}, 360(8):4425–-4439, 2008.

		
		
		
		\bibitem{Bogachev1998}
		V.I.~Bogachev.
		\newblock {\em Gaussian measures}.
		\newblock Mathematical surveys and monographs, 62 (American Mathematical Society, Providence, RI, 1998).
    	
		\bibitem{BurqTzvet}
		N.~Burq and N.~Tzvetkov.
		\newblock Random data {C}auchy theory for supercritical wave equations. {I}.
		{L}ocal theory.
		\newblock {\em Invent. Math.}, 173(3):449--475, 2008.
		
		\bibitem{CG23}
		N.~Camps and L.~Gassot.
		\newblock Pathological set of initial data for scaling-supercritical nonlinear Schrödinger equations.
		\newblock {\em Int. Math. Res. Not.}, 2023(15):13214--13254, 2023.
		
		\bibitem{CG24}
		R.~Carles and L.~Gassot.
		\newblock Pathological set with loss of regularity for nonlinear Schrödinger equations.
		\newblock {\em Ann. Inst. H. Poincaré C, Anal. Non Linéaire}, 42(3):715--753, 2024.
		
		\bibitem{PZ14}
		G.~Da Prato and J.~Zabczyk.
		\newblock {\em Stochastic Equations in Infinite Dimensions}, 				2nd Edition (Cambridge University Press, 2014).
		
		\bibitem{fernique}
		X.~Fernique.
		\newblock Regularité des trajectoires des fonctions aléatoires Gausiennes.
		\newblock {\em Ecole d’Eté de Probabilités de St. Flour IV-1974. Lecture Notes in Mathematics}, 480, 1--96, 1975.
		
		\bibitem{GS64}
		I.M.~Gelfand and G.E.~Shilov.
		\newblock {\em Generalized functions, Vol. I: Properties and Operations}.
		\newblock Academic Press (American Mathematical Society, Providence, RI, 1964).
		

		\bibitem{GV95}
		J.~Ginibre and G.~Velo.
		\newblock Generalized Strichartz Inequalities for the Wave Equation.
		\newblock {\em Partial Differential Operators and Mathematical Physics. Oper. Theory Adv. Appl.}, 78, 1995.
		
		\bibitem{grafakos_kp}
		L.~Grafakos and S.~Oh.
		\newblock The Kato-Ponce inequality.
		\newblock {\em Commun. Partial Differ. Equ.}, 39(6):1128–-1157, 2014.	
		
		
		
		\bibitem{GOTW22}
		T.~Gunaratnam, T.~Oh, N.~Tzvetkov and H.~Weber.
		\newblock Quasi-invariant Gaussian measures for the nonlinear wave equation in three dimensions.
		\newblock {\em Probab. Math. Phys.}, 3(2):343–-379, 2022.
		
		\bibitem{Miyachi1980}
		A.~Miyachi.
		\newblock On some estimates for the wave equation in $L^p$ and $H^p$.
		\newblock {\em J. Fac. Sci. Univ. Tokyo Sect. IA Math.}, 27, 1980.		
		
		
	
	
		
		\bibitem{MS1}
		C.~Muscalu, W.~ Schlag.
		\newblock {\em Classical and Multilinear Harmonic Analysis}.
		\newblock  Cambridge University Press (Cambridge Studies in Advanced Mathematics, 2013).	
		

\bibitem{Peral1980}
		J.C.~Peral.
		\newblock $L^p$ estimates for the wave equation.
		\newblock {\em J. of Funct. Anal.}, 36(1):114--145, 1980.
	
		
		\bibitem{SSS91}
		A. Seeger, C.D. Sogge and E.M.~Stein.
		\newblock {\em Regularity Properties of Fourier Integral Operators}.
		\newblock Ann. Math., 134(2): 231--251, 1991.
		
		\bibitem{Sj70}
		S. Sjöstrand.
		\newblock {\em On the Riesz means of the solutions of the Schrödinger equation}.
		\newblock Ann. Sc. norm. super. Pisa, 24(2): 331--348, 1970.

		\bibitem{soggeNLW}
		C.D.~Sogge.
		\newblock {\em Lectures on Non-Linear Wave Equations}.
		\newblock 2nd Edition (International Press of Boston, Inc. 2008).	
		
		\bibitem{S70}
		E.M.~Stein.
		\newblock {\em Singular Integrals and Differentiability Properties of Functions}.
		\newblock Princeton Mathematical Series. Princeton University Press, 1970.
		
		\bibitem{steinha}
		E.M.~Stein and T.S.~Murphy.
		\newblock {\em Harmonic Analysis (PMS-43): Real-Variable Methods, Orthogonality, and Oscillatory Integrals}.
		\newblock Princeton Mathematical Series (Princeton University Press, 1993).
		
		\bibitem{steinshak-complex}
		E.M.~Stein and R.~Shakarchi.
		\newblock {\em Complex Analysis}.
		\newblock Princeton Lectures in Analysis, No. 2. (Princeton University Press, 2003).
		
		\bibitem{steinweiss}
		E.M.~Stein and G.~Weiss.
		\newblock {\em Introduction to Fourier Analysis on Euclidean Spaces}.
		\newblock Princeton Mathematical Series (Princeton University Press, 1971).
		
		
		\bibitem{ST20}
		C.~Sun and N.~Tzvetkov.
		\newblock Concerning the pathological set in the context of probabilistic well-posedness.
		\newblock {\em C. R. Math.}, 358(9-10):989–-999, 2020.	
		
		\bibitem{GST19}
		M.~Gubinelli, P.~E.~Souganidis, and N.~Tzvetkov.
		\newblock In {\em Singular random dynamics ({C}etraro, 2016)} .
		\newblock {\em Lecture Notes in Math.}, CIME Foundation Subseries, pages 221--313. Springer, 2019.
		
		
		\bibitem{W65}
		S.~Wainger.
		\newblock {\em Special Trigonometric Series in k - dimensions}.
		\newblock Memoirs of the American Mathematical Society, 1965.	
		
		\bibitem{watson}
		G.N.~Watson.
		\newblock {\em A Treatise on the Theory of Bessel Functions}.
		\newblock Cambridge Mathematical Library (Cambridge University Press, 1922).
		
		\bibitem{wolff}
		T.~Wolff.
		\newblock {\em Lectures on Harmonic Analysis}.
		\newblock AMS University Lecture Notes, 2003.
		
	\end{thebibliography}
\end{document}